\documentclass[11pt]{article}
\usepackage{amssymb, amsmath, amsthm, xstring, mathrsfs}

\newcommand{\spa}[1]{
	\IfEqCase{#1}{ 
		{1}{\:}
		{2}{\:\:}
		{3}{\:\:\:}
		{4}{\:\:\:\:}
		{5}{\:\:\:\:\:}
		{6}{\:\:\:\:\:\:}
		{7}{\:\:\:\:\:\:\:}
		{8}{\:\:\:\:\:\:\:\:}
		{9}{\:\:\:\:\:\:\:\:\:}
		{10}{\:\:\:\:\:\:\:\:\:\:}
		{11}{\:\:\:\:\:\:\:\:\:\:\:}
		{12}{\:\:\:\:\:\:\:\:\:\:\:\:}
		{13}{\:\:\:\:\:\:\:\:\:\:\:\:\:}
		{14}{\:\:\:\:\:\:\:\:\:\:\:\:\:\:}
		{15}{\:\:\:\:\:\:\:\:\:\:\:\:\:\:\:}
		{16}{\:\:\:\:\:\:\:\:\:\:\:\:\:\:\:\:}
		{17}{\:\:\:\:\:\:\:\:\:\:\:\:\:\:\:\:\:}
		{18}{\:\:\:\:\:\:\:\:\:\:\:\:\:\:\:\:\:\:}
		{19}{\:\:\:\:\:\:\:\:\:\:\:\:\:\:\:\:\:\:\:}
		{20}{\:\:\:\:\:\:\:\:\:\:\:\:\:\:\:\:\:\:\:\:}
		{-1}{\!}
		{-2}{\!\!}
		{-3}{\!\!\!}
		{-4}{\!\!\!\!}
		{-5}{\!\!\!\!\!}
	}[\PackageError{spa}{Undefined option to spa: #1}{}]
}

\newcommand{\subST}{_{_{ST}}}
\newcommand{\subS}{_{^{\mathbb{S}^3}}}

\newcommand{\kauf}[1]{\langle #1 \rangle\subST}
\newcommand{\kaufS}[1]{\langle #1 \rangle}

\usepackage{graphicx}

\theoremstyle{plain}
\newtheorem{thm}{Theorem}
\newtheorem{lmm}{Lemma}
\newtheorem{corol}{Corollary}
\newtheorem{prop}{Proposition}

\theoremstyle{definition}
\newtheorem{defn}{Definition}

\newtheorem{exmp}{Example}
\theoremstyle{remark}

\newtheorem*{nrmrk}{Remark}
\newtheorem*{nthm}{Theorem}
\newtheorem*{nprop}{Proposition}

\date{}
\title{\textbf{On lassos and the Jones polynomial of satellite knots}}
\date{} 
\author{Adri\'an Jim\'enez Pascual}

\begin{document}
\maketitle

\begin{align*}
\\\\\\
&\includegraphics[scale=0.55]{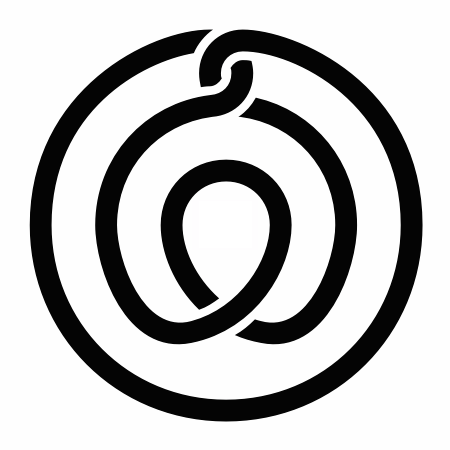}\\\\\\\\\\\\
\end{align*}

\begin{abstract}
In this master thesis, I present a new family of knots in the solid torus called \emph{lassos}, and their properties. Given a knot $K$ with Alexander polynomial $\Delta_K(t)$, I then use these lassos as patterns to construct families of satellite knots that have Alexander polynomial $\Delta_K(t^d)$ where $d\in\mathbb{N}\cup \{0\}$. In particular, I prove that if $d\in\{0,1,2,3\}$ these satellite knots have different Jones polynomials. For this purpose, I give rise to a formula for calculating the Jones polynomial of a satellite knot in terms of the Jones polynomials of its pattern and companion.
\end{abstract}

\newpage

\section{Introduction}
The Alexander polynomial was the first to be discovered and is the best known polynomial invariant for knots. It was developed in 1923 by J. W. Alexander. It was not until 1969 that J. Conway introduced a modified version of the Alexander polynomial, now called the Conway polynomial, that used a skein relation to its computation. He also presented a normalization of the Alexander polynomial, so that the Alexander polynomial of a knot $K$ would be given by an equality rather than by an equality ``up to multiplication by a unit'' of the $\mathbb{Z}[t^{-1},t]$-module $H_1(K_\infty;\mathbb{Z})$. In 1984 V. Jones discovered the so called Jones polynomial, which also can be obtained through a skein relation. In this paper, we will write $\Delta_K(t)$ for the Conway-normalized Alexander polynomial of a knot $K$ and $J(K)$ for its Jones polynomial.\\

None of the above mentioned polynomials can completely distinguish knots. It has been proven by many authors (\cite{Cr,Mo}) that there exist infinitely large families of knots sharing their polynomials. In concrete, \cite{Kan} proved that there exist infinitely many knots with the same Alexander polynomial and different Jones polynomial, as well as infinitely many knots with the same Jones polynomial and different Alexander polynomial.\\

Some other results of interest which serve as a reference point to this work are the following well-known equalities, for which we consider $K_1$ and $K_2$ to be two knots in $\mathbb{S}^3$, and we write their connected sum as $K_1\#K_2$.
\[\Delta_{K_1\#K_2}(t)=\Delta_{K_1}(t)\Delta_{K_2}(t)\]
\[J(K_1\#K_2)=J(K_1)J(K_2)\]

We will also make use of the following theorem, since we will be working with the structure and creation of satellite knots.

\begin{nthm}[\cite{Li}]\label{lick}
Let $P$ be a knot in an unknotted solid torus, $C$ be a knot in $\mathbb{S}^3$ and $Sat(P,C)$ be their satellite knot. Then, the following equality holds.
\[\Delta_{Sat(P,C)}(t)=\Delta_P(t)\Delta_C(t^n),\]
where $P$ represents $n$ times a generator of $H_1(ST)$.
\end{nthm}

In this paper we will introduce a new family of knots in the solid torus, which we will call \emph{lassos}. We will also set some properties of these lassos, and we will use them to construct satellite knots with the same Alexander polynomial. In particular, we will prove that the Alexander polynomial of a satellite knot which uses a lasso as pattern will look like the following.

\begin{nprop}
Let $L$ be a lasso of degree $d$, $C$ be a knot in $\mathbb{S}^3$ and $Sat(L,C)$ be their satellite knot. Then, the following equality holds.
\[\Delta_{Sat(L,C)}(t)=\Delta_C(t^d).\]
\end{nprop}

Delving deeper into the structure of satellite knots, we will also present some new general results regarding these. In concrete, it will be proven in this paper that the Jones polynomial of a satellite knot can be expressed as follows.

\begin{nthm}
Let $P$ be a knot in an unknotted solid torus, $C$ be a knot in $\mathbb{S}^3$ and $Sat(P,C)$ be their satellite knot. Then, the following equality holds.
\[J(Sat(P,C))=J\subST(P)\Big|_{z^k\subST=J(C;\:k)},\]
where $J\subST(P)$ represents the Jones polynomial in the solid torus and $J(C;k)$ the $k$-parallel Jones polynomial.\\
\end{nthm}

Making use the results here presented, we will construct two infinite families of knots using lassos. All the members in each family share their Alexander polynomial. Yet, we will use the above result to prove the next two theorems.

\begin{nthm}
Consider the family of lassos $L(r)$ and let $C$ be a knot in $\mathbb{S}^3$ such that $\Delta_C(t)\neq 1$ and $J(C;2)\neq J(U;2)$. Then, $Sat(L(r_1),C)$ and $Sat(L(r_2),C)$ are different knots for $r_1\neq r_2$.
\end{nthm}

\begin{nthm}
Consider the family of lassos $L(1,r)$ and let $C$ be a knot in $\mathbb{S}^3$ such that $\Delta_C(t)\neq 1$ and $J(C;3)\neq J(C)J(U;3)$. Then, $Sat(L(1,r_1),C)$ and $Sat(L(1,r_2),C)$ are different knots for $r_1\neq r_2$.\\
\end{nthm}

Some results had already been given regarding knots that have the same Alexander polynomial and different Jones polynomial (\cite{Kan,Kaw}). Yet it is the first time that this kind of result is constructed using proper satellite knots and that a formula for both polynomials is explicitly given.\\

This paper will be organized in four sections. In the first section we will define the concept of lasso and its degree. We will then prove some isotopy equivalences between lassos, and provide the first result regarding satellite knots, which corresponds to the Alexander polynomial of satellite knots using lassos as patterns.

In the next section we will present the Kauffman bracket skein module of a knot in the solid torus, and we  will use it to calculate Jones polynomials using that as basis. We will then proceed to construct in a certain manner diagrams of satellite knots, and we will make explicit formulae for calculating the Kauffman bracket of such diagrams and the Jones polynomial of satellite knots.

In the third section, we will make use of sections 1 and 2 to give rise to the last two theorems above and some other results regarding those families of knots.

Finally, we will proceed in the last section to construct an example going through all the concepts and results presented in this paper.

\subsection*{Acknowledgements}
I would like to thank my advisor, Professor T. Kohno, for his dedication and support to my reseach, keeping me always busy with new goals to look for and achieve. The help of the people around me, specially T. Kitayama, C. Moraga and Y. Nozaki, was also indispensable in the process of giving substance to this master thesis. I would also like to thank professors A. Kawauchi and T. Kanenobu for enlighting me with references and descriptions of results that helped me to better understand the results developed in this paper, and A. Stoimenow for also sharing his ideas with me. Last but not least, thanks also to Y. Zhang for her continuous encouragement and understanding.

\newpage

\section{Lassos, satellite knots and the Alexander polynomial}

This first section will serve as introductory to the concept of \emph{lasso} and its \emph{degree}. In the later part of the section we will also relate the Alexander polynomial of satellite knots using lassos to this \emph{degree}.

\subsection{Lassos}

The following definitions will be considered inside the solid torus $ST \simeq \mathbb{S}^1\times \mathcal{D}^2$, and we will depict them on the annular projection of $ST$. The unknots (trivial knots) that we will mention will be isotopic to $\mathbb{S}^1\times \{0\}$ in $ST$ (unless otherwise specified). 

\begin{defn}\label{simple-lasso}
We call a \emph{simple lasso} $L(r)$ to the $r$-twisted knot sum of two nested unknots, as shown below.\\
\begin{align*}
&\includegraphics[scale=0.3]{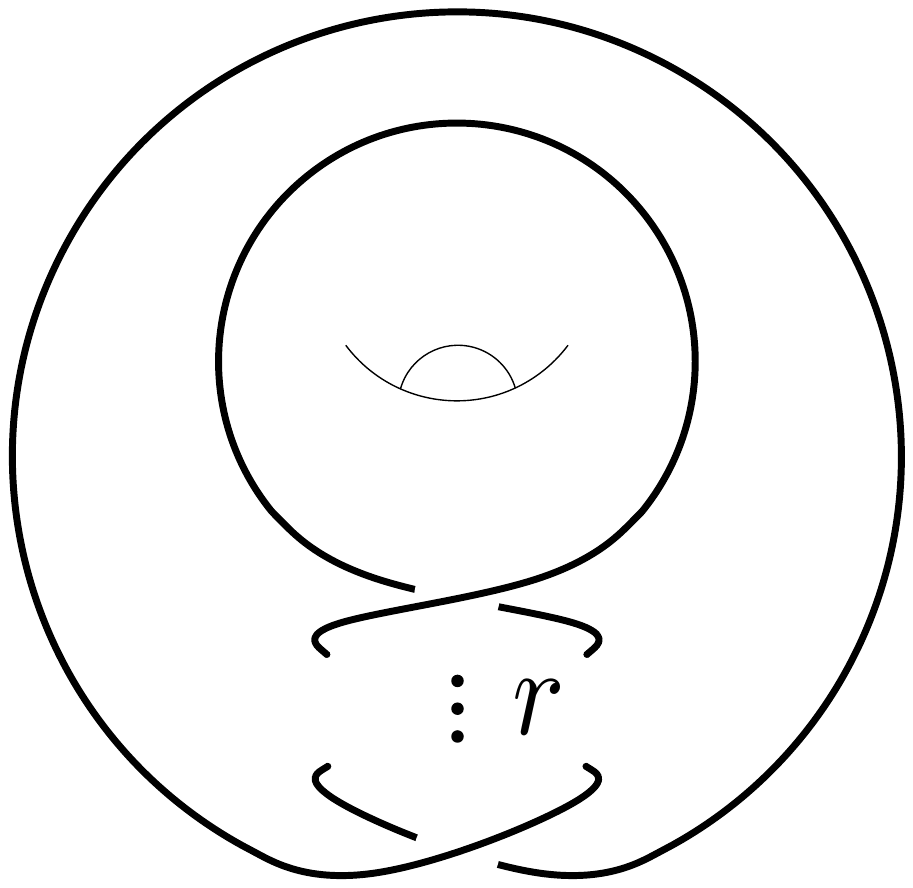}
&&\includegraphics[scale=0.3]{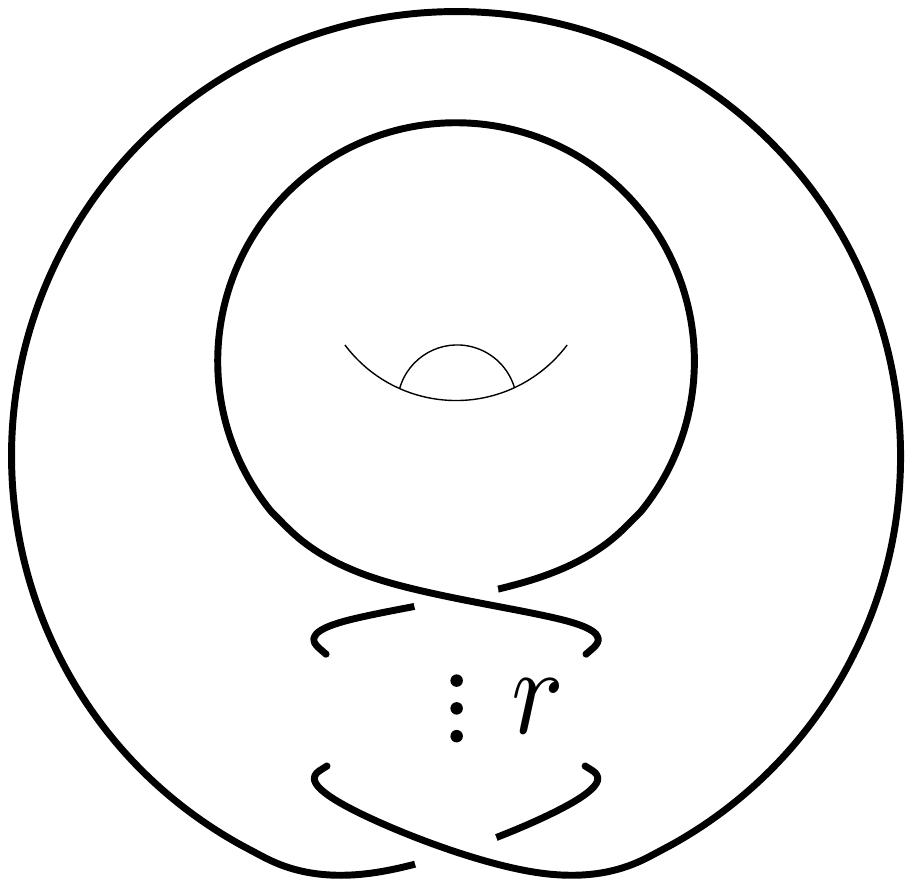}\\
&\spa{11} r>0
&&\spa{11} r<0
\end{align*}
\begin{center}
(the \includegraphics[scale=0.3]{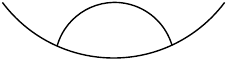} symbol should be understood as the center of $ST$)
\end{center}

Please notice that $r$ can be both positive or negative depending on the rotation direction of the twists, as shown in the pictures. The case $r=0$ will hence represent the standard (untwisted) knot sum.
\begin{align*}
&\includegraphics[scale=0.3, angle=90]{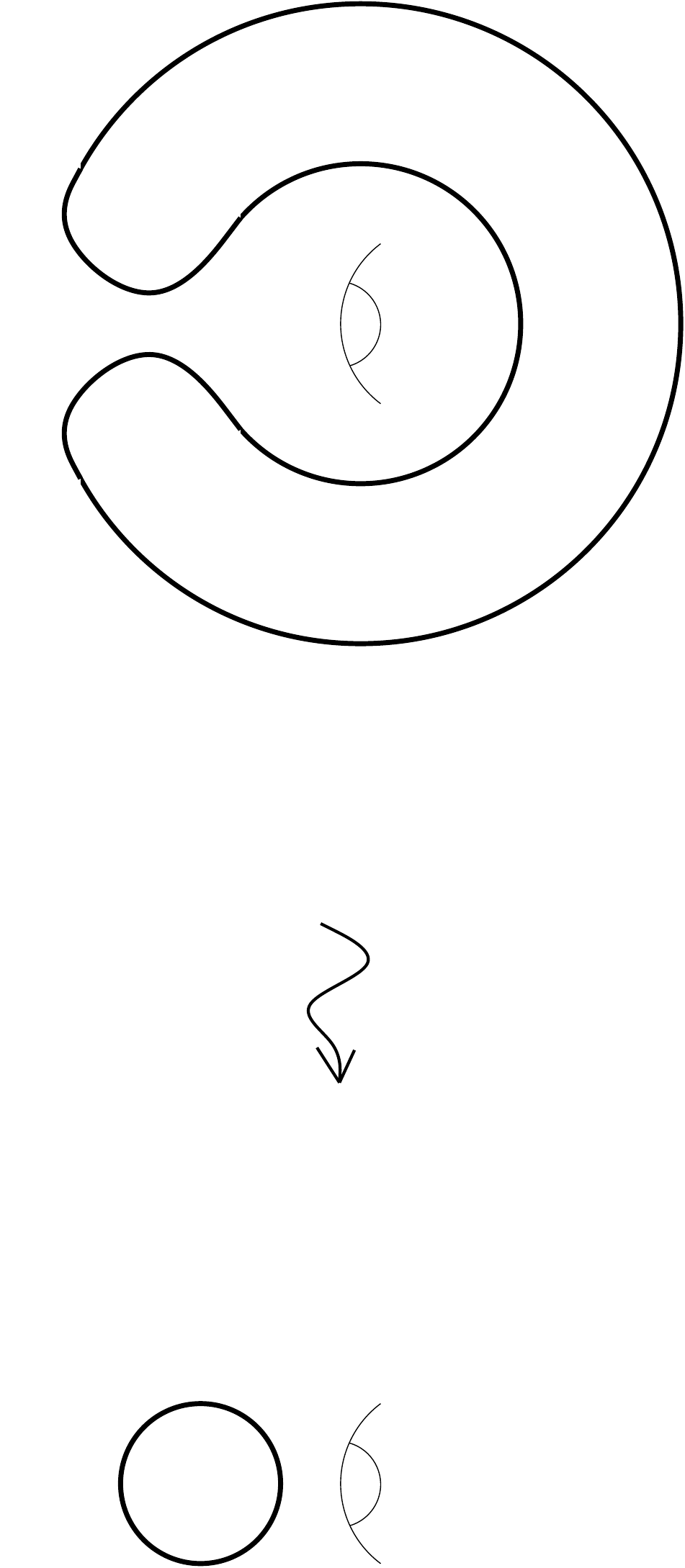}\\
&\spa{12} L(0) \spa{20}\spa{20}\spa{8} L(0)
\end{align*}
\end{defn}

\begin{nrmrk}
Using this notation, a \emph{Whitehead double} can be written as $L(\pm{2})$.\\
\begin{align*}
&\includegraphics[scale=0.3]{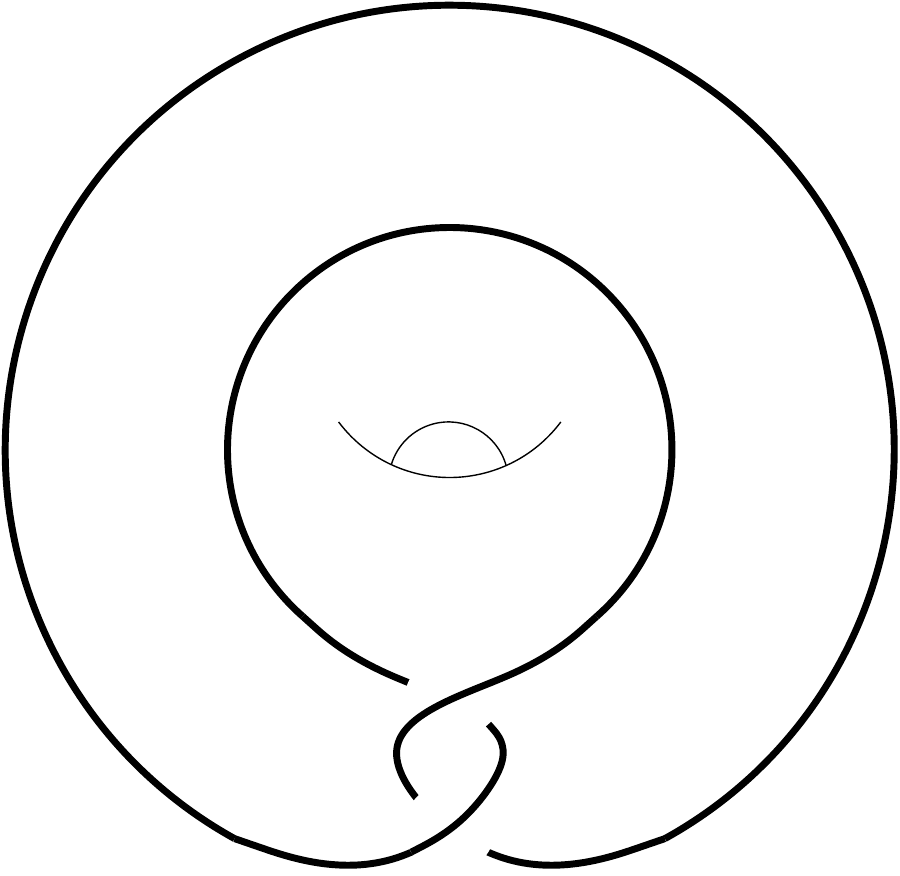}
&&\includegraphics[scale=0.3]{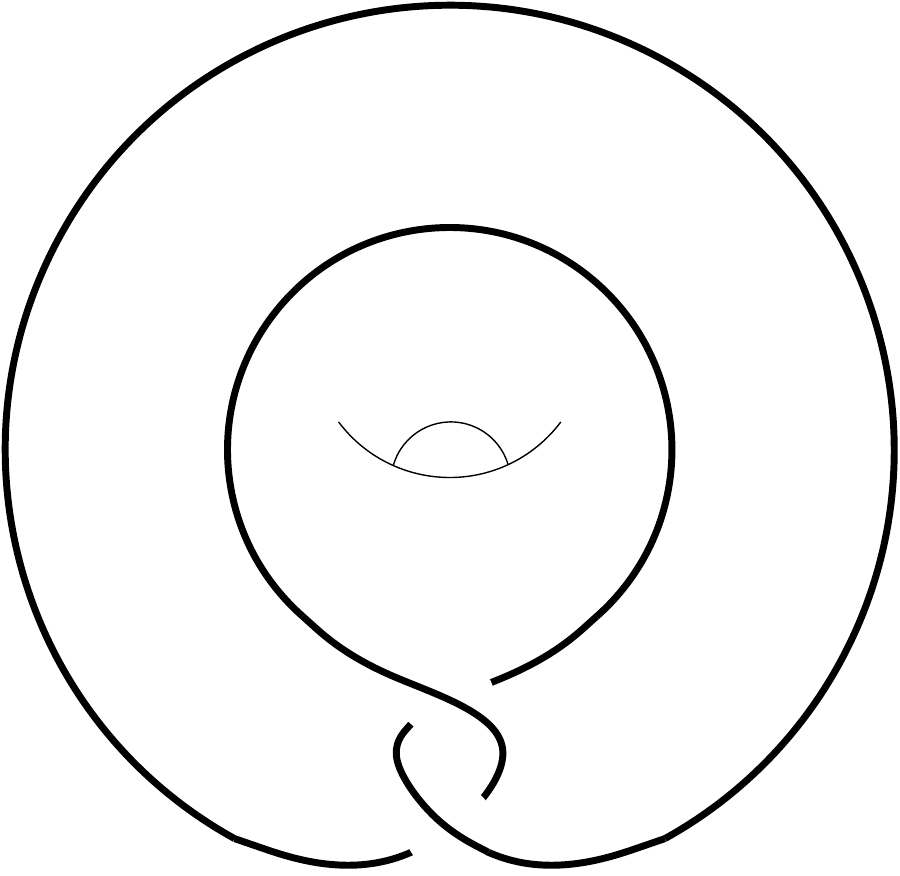}\\
&\spa{11} L(2)
&&\spa{11} L(-2)
\end{align*}
\end{nrmrk}

We now generalize this definition to the following.

\begin{defn}\label{general-lasso}
We call a \emph{lasso} $L(r_1,r_2,...,r_m)$ to the consecutive $r_i$-twisted knot sum of $m+1$ nested unknots. 
\begin{align*}
&\includegraphics[scale=0.3]{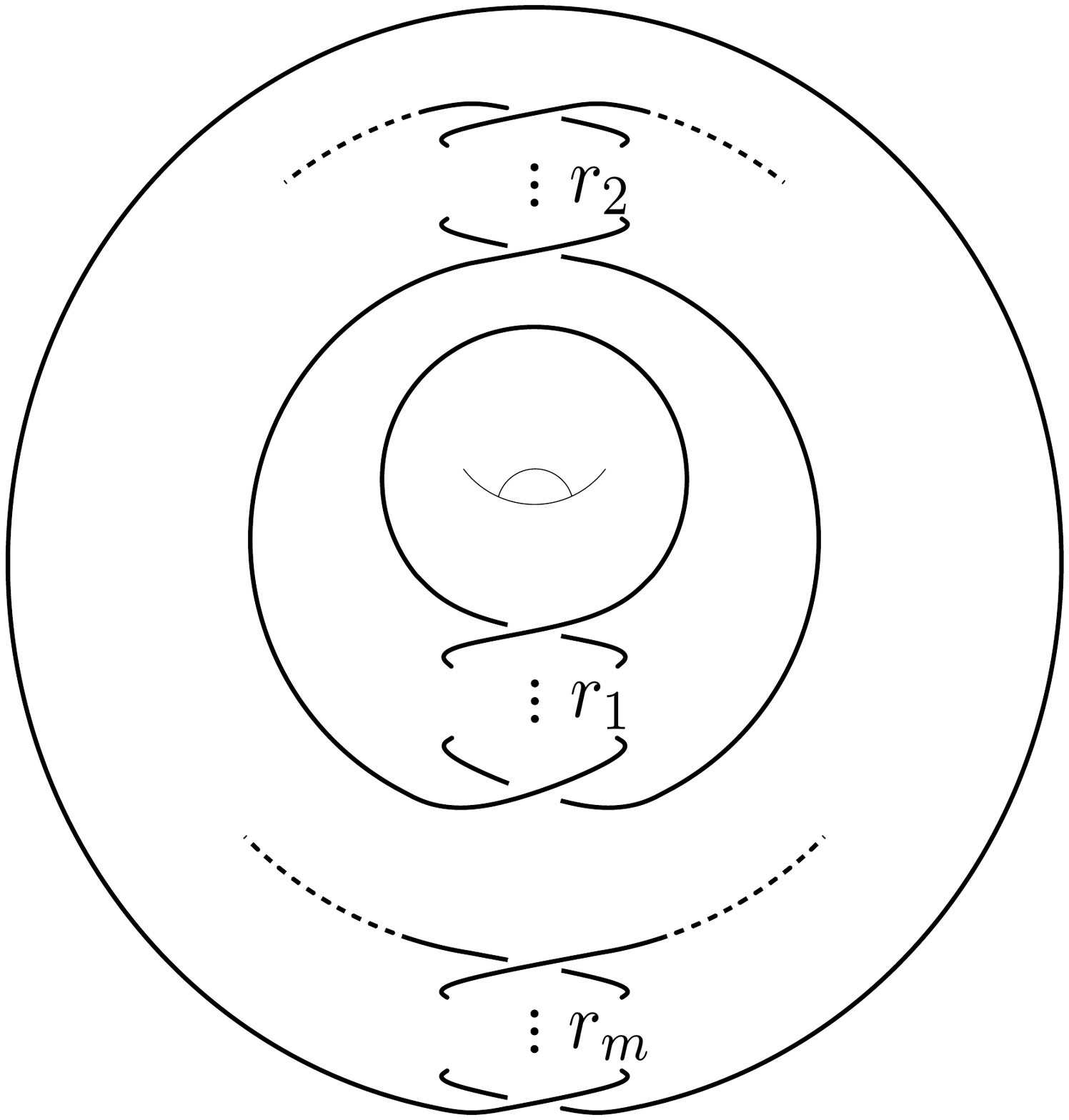}\\
&\spa{14} L(r_1,r_2,...,r_m)
\end{align*}
\end{defn}

These depictions of lassos over the annular projection of $ST$ will be referred to as \emph{normal diagrams}, and will be used throughout this paper. Let us draw some concrete examples of lassos with normal diagrams.
\begin{align*}
&\includegraphics[scale=0.3]{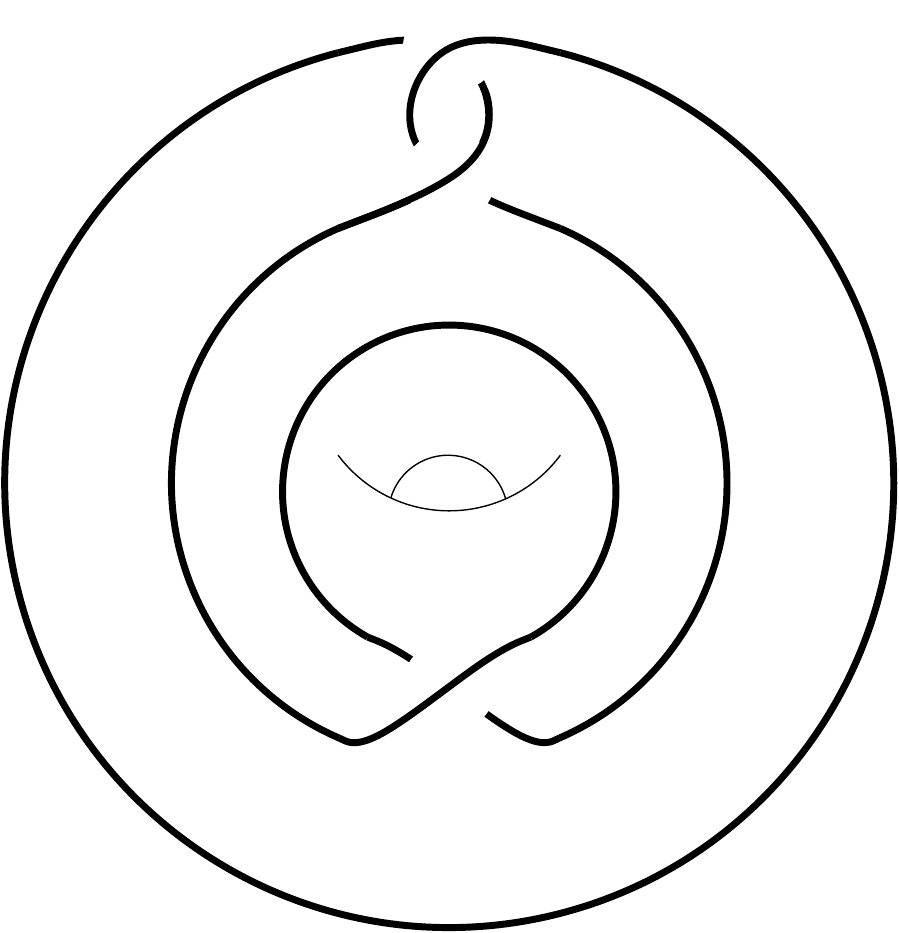}
&&\includegraphics[scale=0.3]{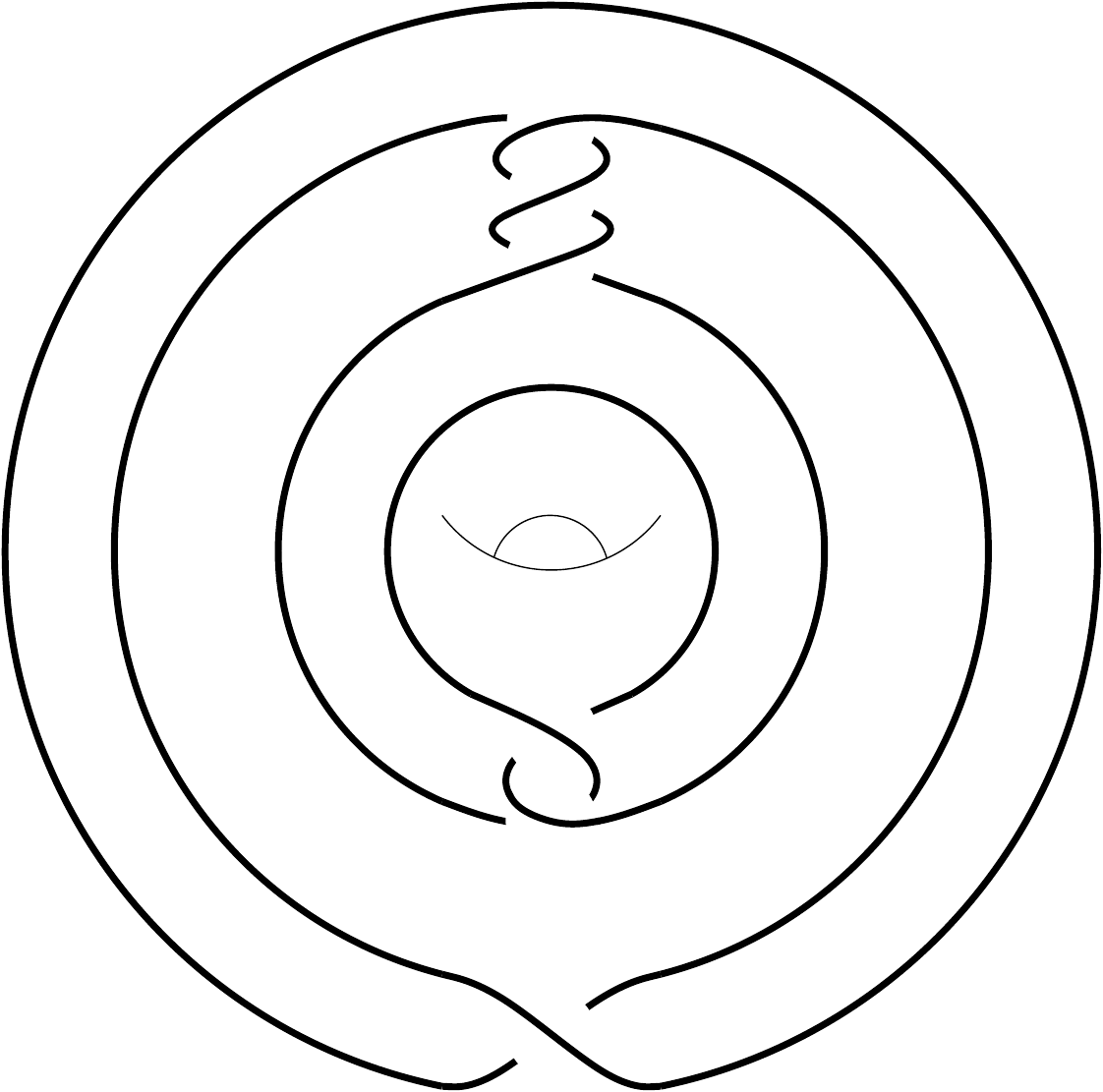}\\
&\spa{9} L(1,2)
&&\spa{8} L(-2,3,-1)
\end{align*}

\begin{nrmrk}
It is convenient to note that, in concordance with the definition of lasso, $L(\emptyset)$ would correspond to only one unknot around the center of the solid torus, and therefore we can depict it as shown under this line.
\begin{align*}
&\includegraphics[scale=0.3]{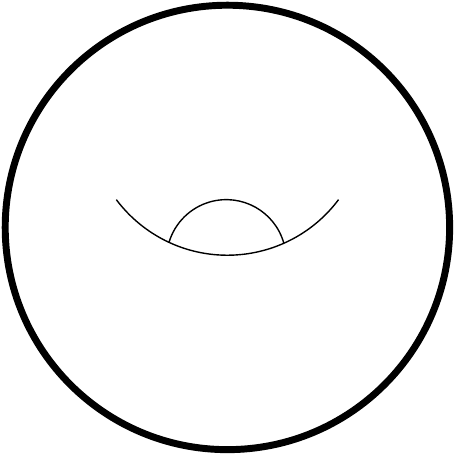}\\
&\spa{4} L(\emptyset)
\end{align*}

However, we will \underline{not} consider any $r_i$ to be $0$, since this would lead to simplifiable cases where no $0$ appears. The applicable simplifications are the following.
\begin{itemize}
\item $L(0,r_2,r_3,...,r_m)\simeq L(r_3,...,r_m)$.
\item $L(r_1,...,r_{i-1},0,r_{i+1},...,r_m)\simeq L(r_1,...,r_{i-1}+r_{i+1},...,r_m)$.
\item $L(r_1,...,r_{m-2},r_{m-1},0)\simeq L(r_1,...,r_{m-2})$.\\

The first two isotopy equivalences are portrayed below.
\end{itemize}
\begin{align*}
&\includegraphics[scale=0.3, angle=90]{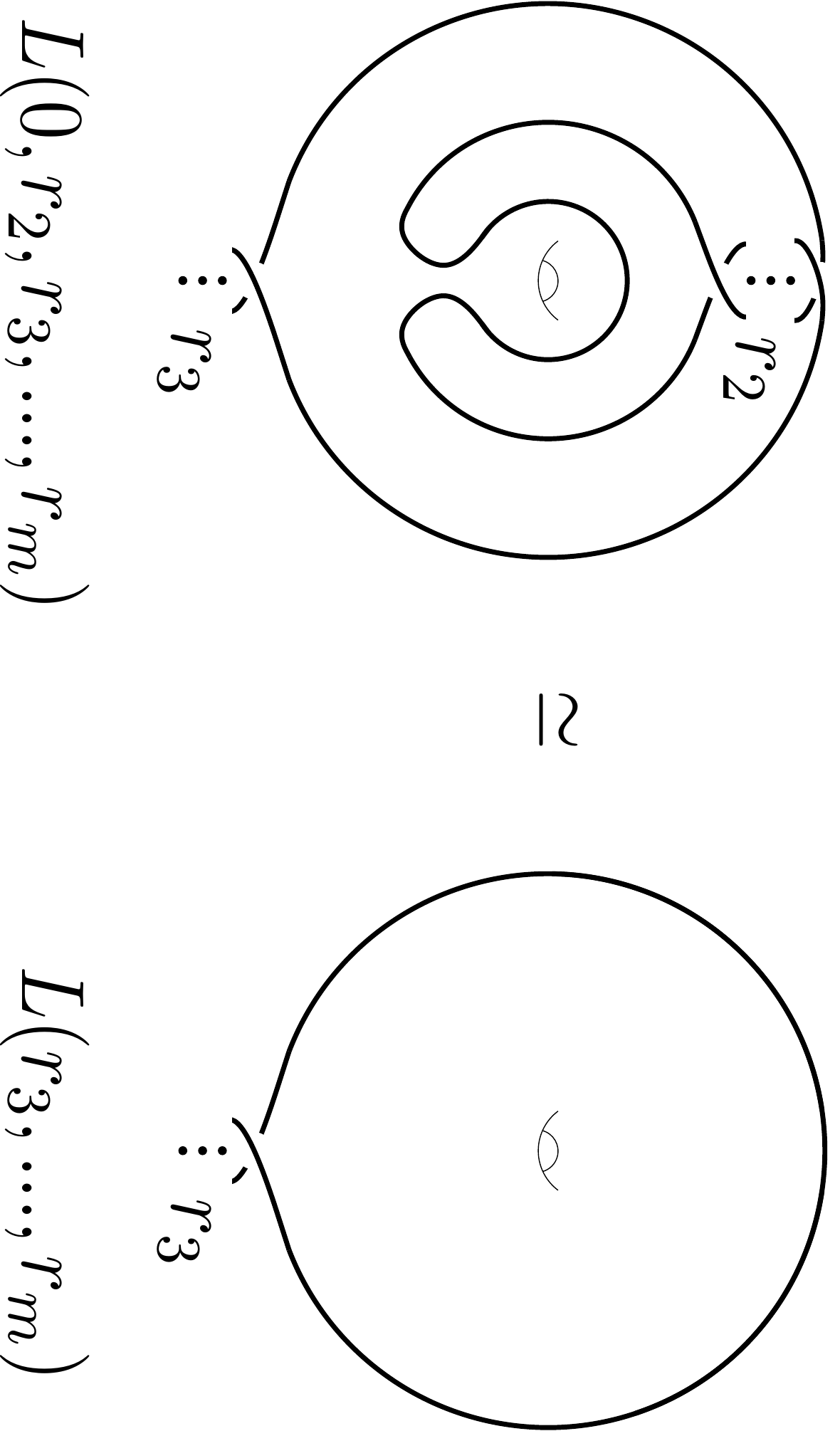}
&&\includegraphics[scale=0.3, angle=90]{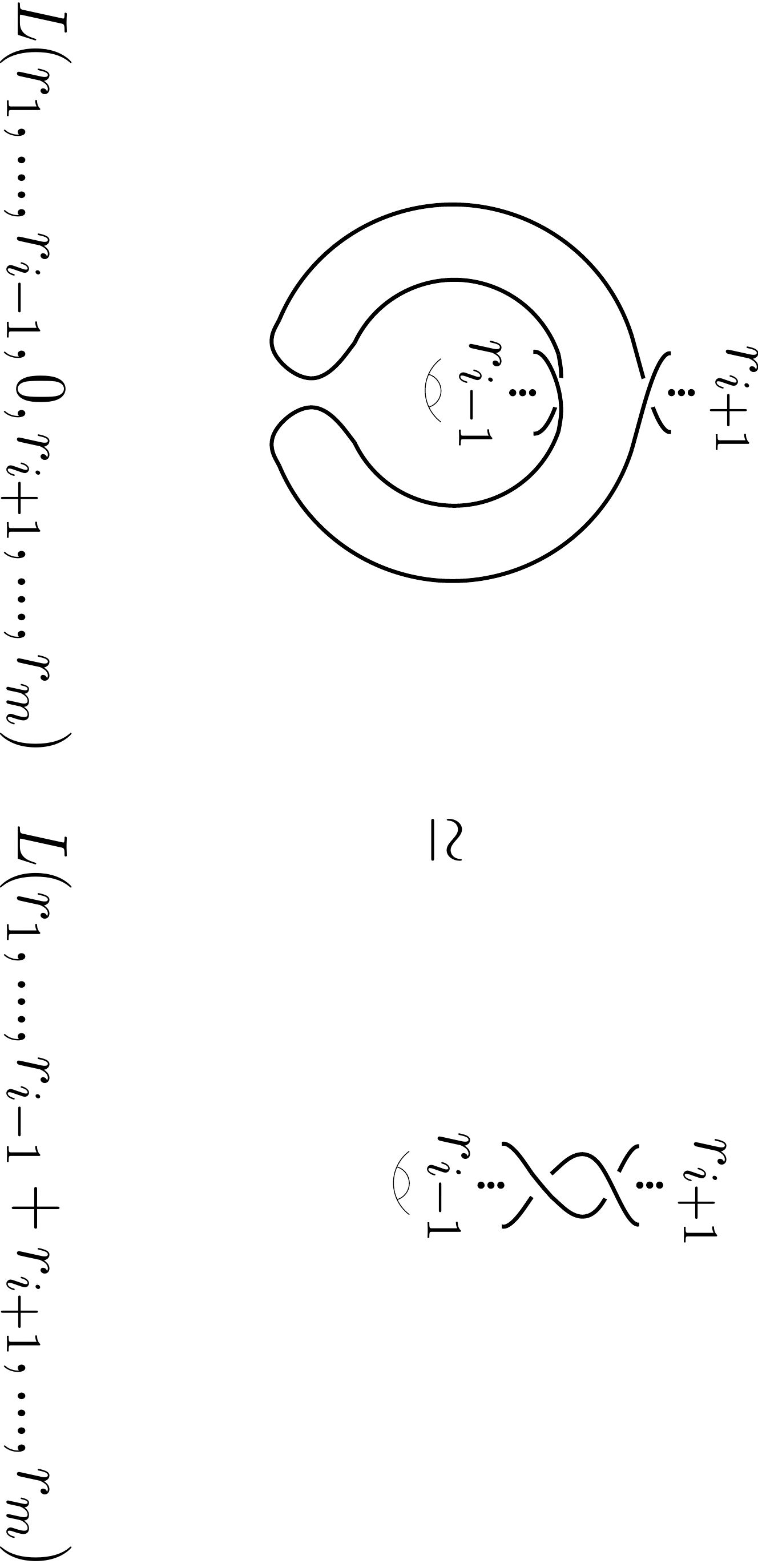}
\end{align*}

The third equivalence is an immediate consequence of the following result.
\end{nrmrk}

\begin{prop}
Let $L(r_1,r_2,...,r_{m-1},r_m)$ be a lasso. Then, we obtain the following isotopy equivalence.
\[L(r_m,r_{m-1},...,r_2,r_1)\simeq L(r_1,r_2,...,r_{m-1},r_m).\]
\end{prop}

\begin{proof}
To simplify drawings, let us represent with boxes where sets of crossings should be.
\[\includegraphics[scale=0.7]{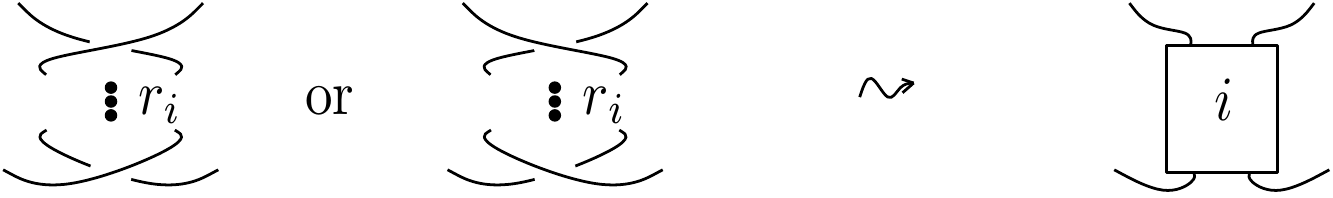}\]

We can draw now our initial lasso $L(r_1,r_2,...,r_m)$ as follows.
\[\includegraphics[scale=0.3]{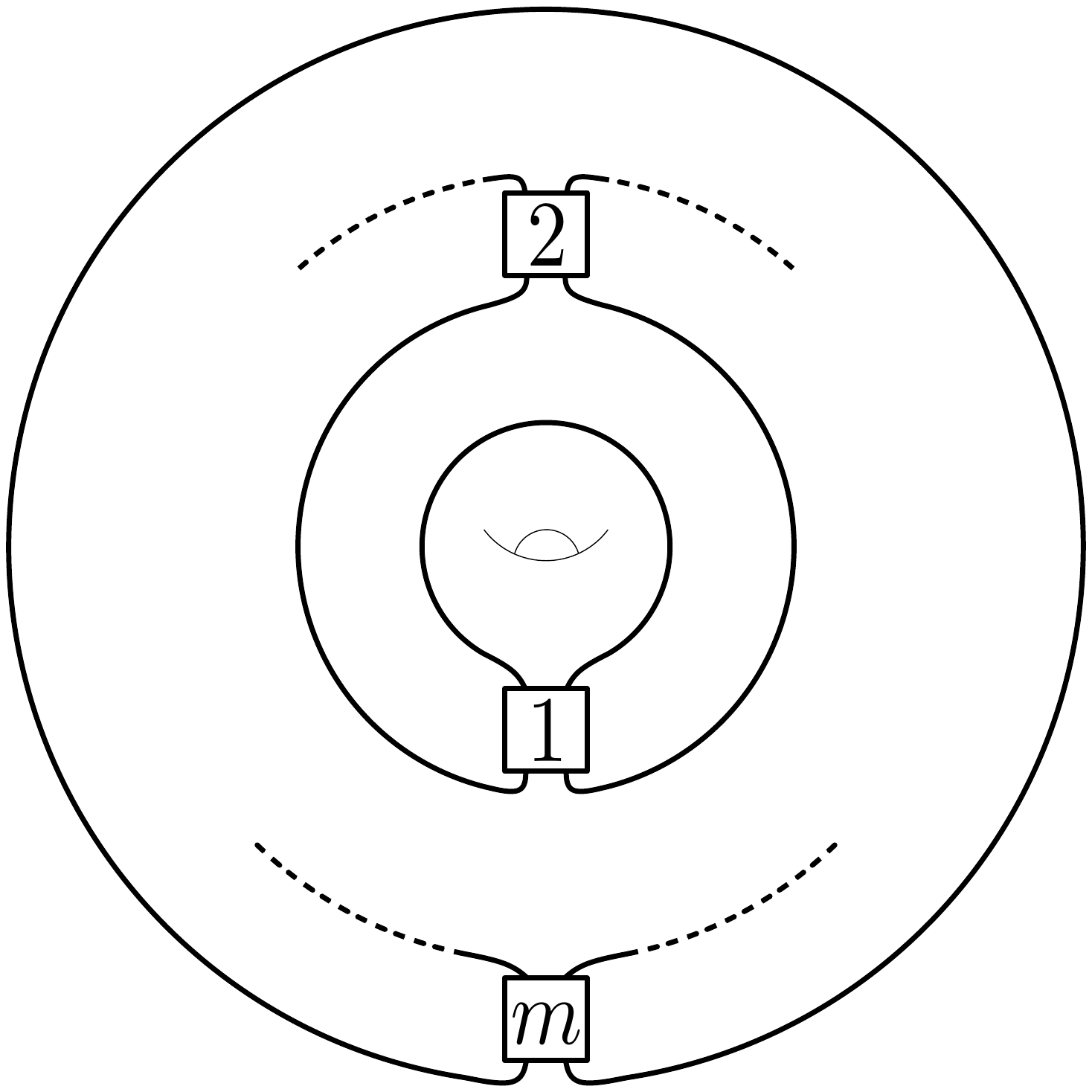}\]

In order to get to $L(r_m,...,r_2,r_1)$ isotopically, we start by flipping the first crossing set around its own horizontal axis, as shown in the first step of the steps sequence. Observe that the piece labelled with $1$ is reversed. Please also notice that flipping a crossing set does not change the number neither the orientation of the crossings inside. Then, we move the flipped crossing set over all the odd-numbered labels to the outmost part, as represented in the second step. In the third image, we have flipped the second crossing set so that the arcs between \raisebox{\depth}{\scalebox{1}[-1]{$1$}} and \raisebox{\depth}{\scalebox{1}[-1]{$2$}} appear above the others. Then, push \raisebox{\depth}{\scalebox{1}[-1]{$2$}} over all the even-numbered labels to the outmost part, but confined by the outer boundary coming from \raisebox{\depth}{\scalebox{1}[-1]{$1$}}. As we see, this process allows us to move the inner crossing sets to the exterior in an orderly fashion, taking care not to tangle what we already have. In the next picture, we see the last flip that has to be performed. Only $m$ remains to be flipped. After the flip, we would have completely reversed the lasso $L(r_1,r_2,...,r_{m-1},r_m)$. Recalling that flipping a crossing set does not change the orientation of the crossings in it, the resulting figure is $L(r_m,r_{m-1},...,r_2,r_1)$.

\begin{align*}
&\spa{12}\includegraphics[scale=0.3]{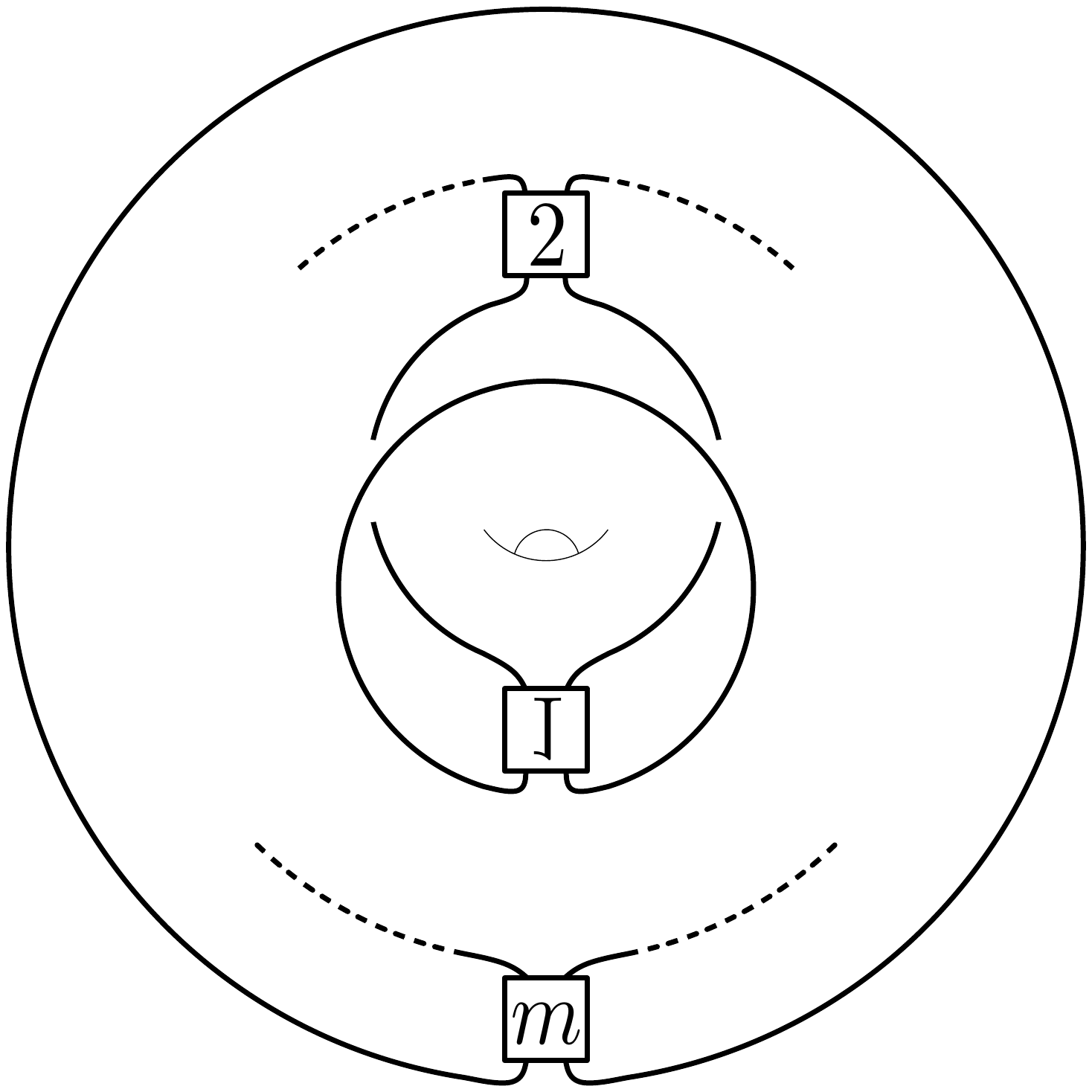}&&\raisebox{5.5 em}{\includegraphics[scale=0.5, angle=90]{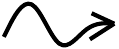}} &&\raisebox{-2.8 em}{\includegraphics[scale=0.6]{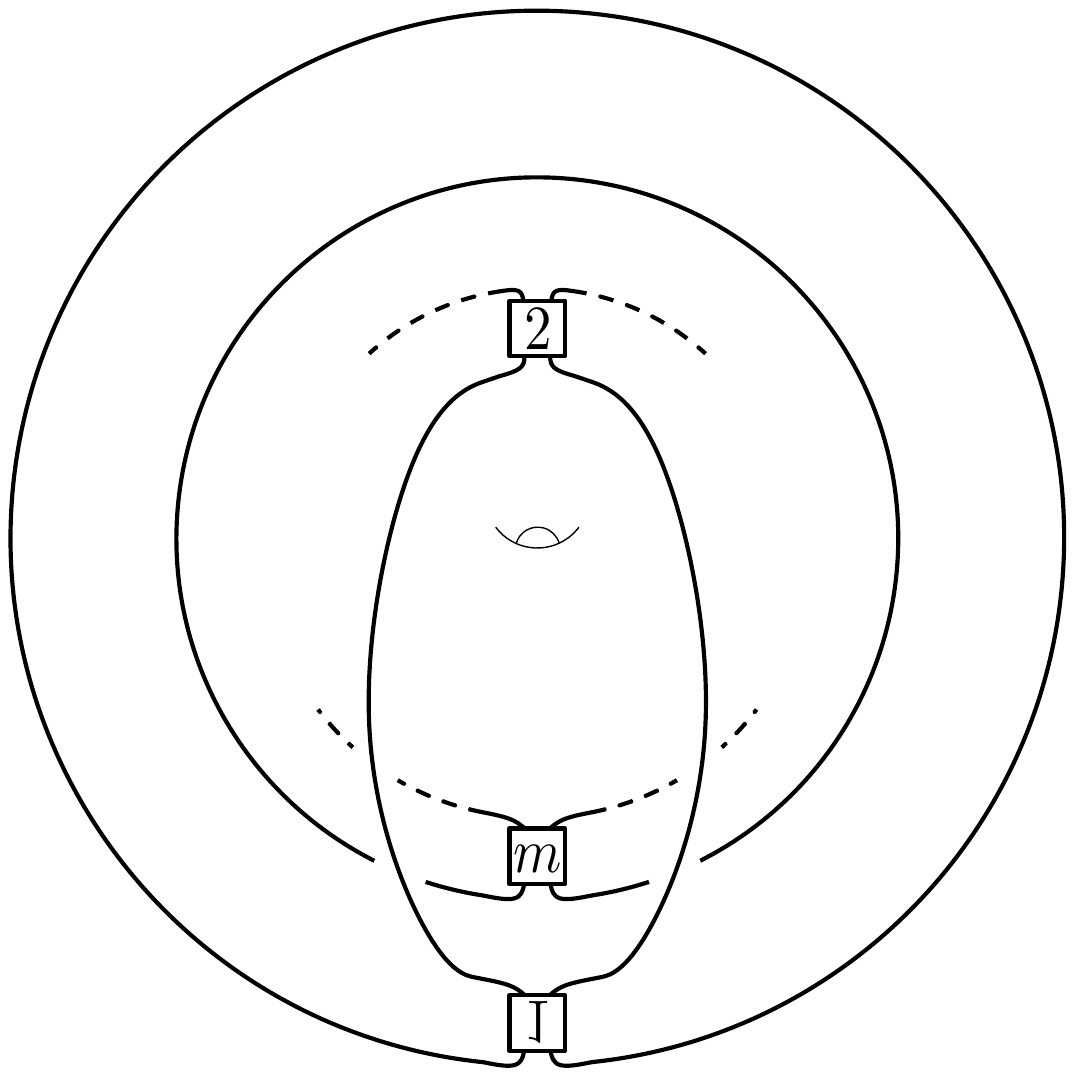}}\\
&&&\includegraphics[scale=0.5, angle=315]{Arrow_down}&&\\
&\includegraphics[scale=0.6]{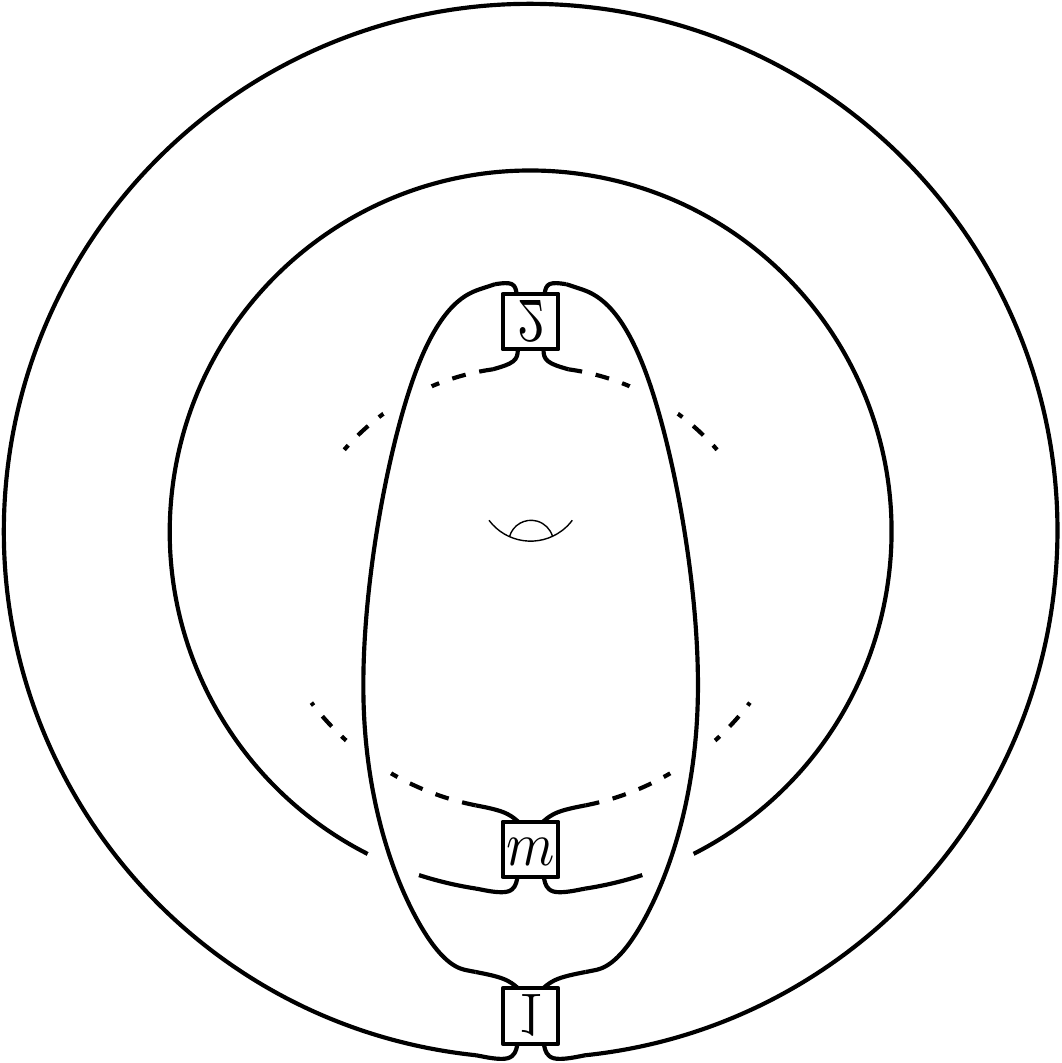} &&\raisebox{8 em}{\includegraphics[scale=0.5, angle=90]{Arrow_down}}&&\includegraphics[scale=0.6]{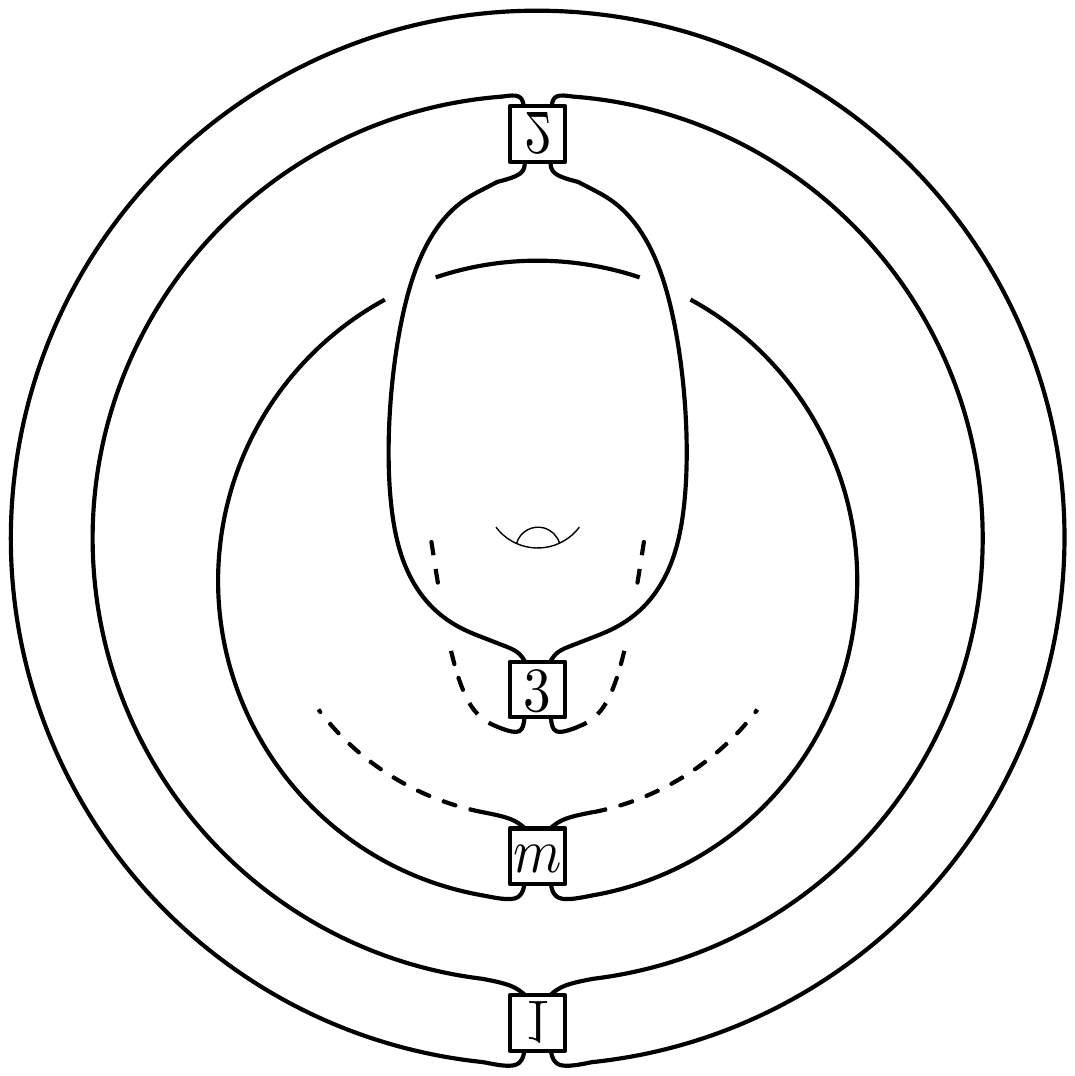}\\
&&&&&\spa{10}\includegraphics[scale=0.5, angle=315]{Arrow_down}\\
\end{align*}
\[\includegraphics[scale=0.6]{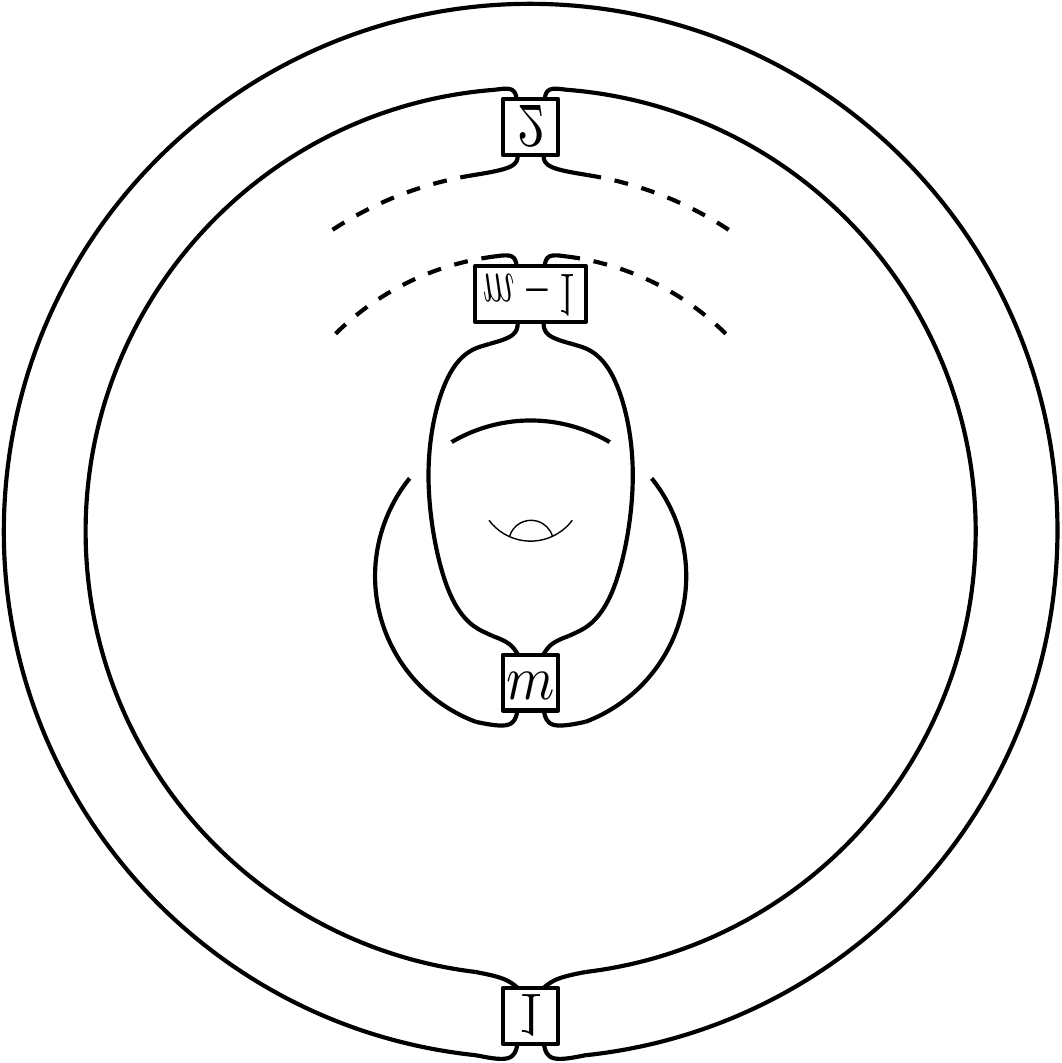}\]
\[\includegraphics[scale=0.3]{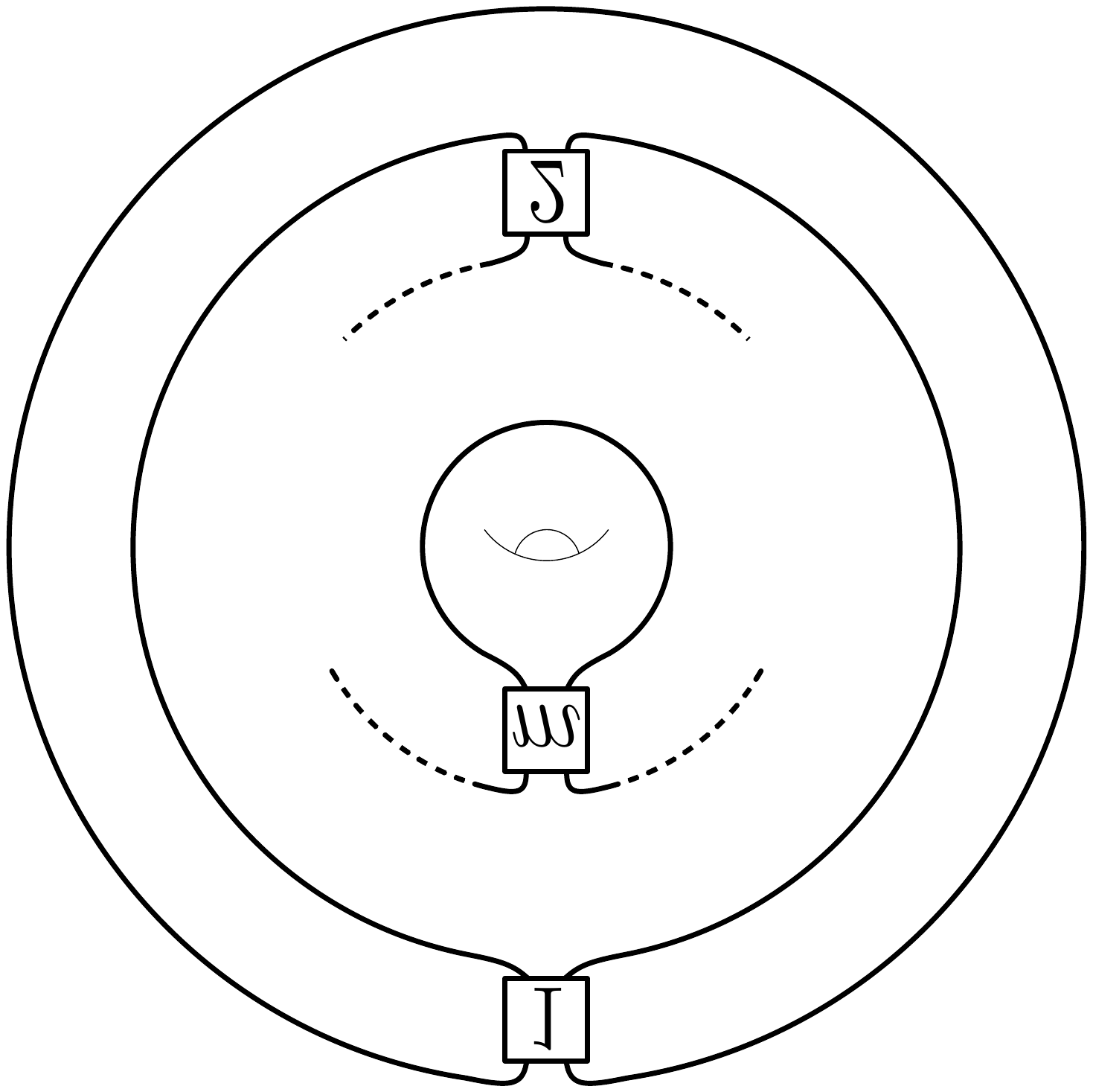}\qquad\quad\includegraphics[scale=0.3]{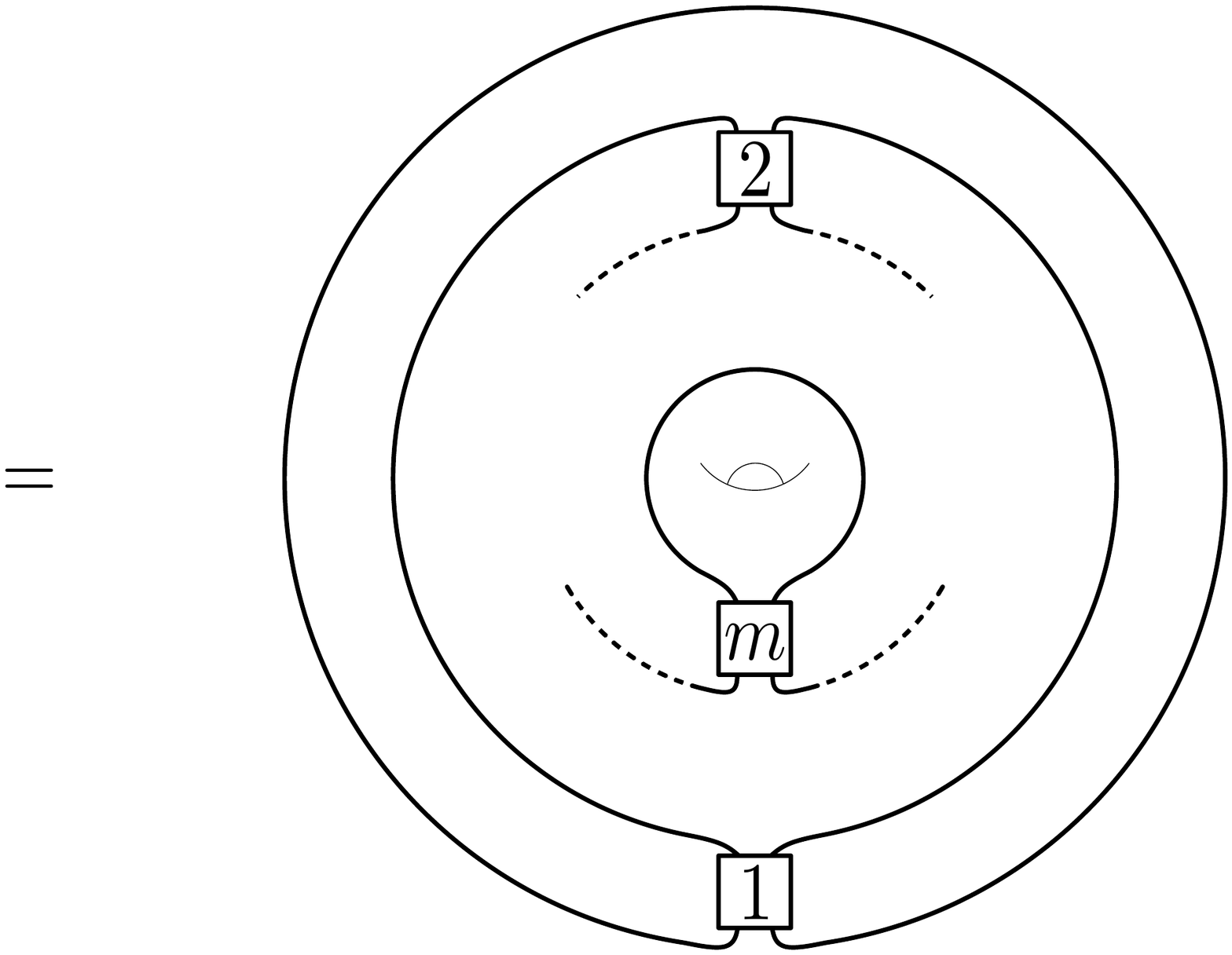}\]
\end{proof}

\begin{defn}\label{degree}
Let $L(r_1, r_2,...,r_m)$ be a lasso. Let us consider the $(m+1)$-tuple $(\bar{r}_0,\bar{r}_1,...,\bar{r}_m)$, whose values are taken according to the next formulae.
\begin{align*}
&\bar{r}_0 = 1,\\
&\bar{r}_i = \left\{ 
\begin{array}{ll}
	(-1)^{1+r_i} & \textrm{if $\bar{r}_{i-1}=1$}\\
	1 & \textrm{if $\bar{r}_{i-1}= -1$}
\end{array}
\spa{10}i\in {\{1,...,m\}}.\right.
\end{align*}

Then, we define the \emph{degree} of $L(r_1,r_2,...,r_m)$ as the sum of the values of the $(m+1)$-tuple.
\[\deg\big(L(r_1,r_2,...,r_m)\big)=\sum\limits_{i=0}^m \bar{r}_i.\]
\end{defn}

\begin{defn}\label{sets}
We define $\mathscr{L}$ and $\mathscr{L}_d$ to be the following sets.
\[\mathscr{L} = \{L\in Emb(\mathbb{S}^1,ST)\;|\;L\textnormal{ is a lasso}\},\]
\[\mathscr{L}_n = \{L\in\mathscr{L}\;|\;\deg(L)=n\}.\]
\end{defn}

Here are some elementary observations:
\begin{enumerate}
\item If $m$ is odd, the degree of a lasso $L(r_1,r_2,...,r_m)$ is always non-negative and at most $m+1$, namely, we have $0\leq \deg\big(L(r_1,r_2,...,r_m)\big)\leq m+1$. This follows from the definition of degree. On the other hand, if $m$ is even, then it forces the degree to be strictly positive, which translates as $1\leq \deg\big(L(r_1,r_2,...,r_m)\big)\leq m+1$. In this case the degree is strictly positive since $\bar{r}_0$ is positive and for the remaining $m$-tuple the first inequality applies.

\item The writhe of an oriented lasso is given by $wr\big(L(r_1,r_2,...,r_m)\big)=-\sum\limits^m_{i=1}r_i$.
\end{enumerate}

\subsection{Satellite knots and the Alexander polynomial}

We will now recall the definitions of framing and satellite knot and bring up a theorem that relates the Conway-normalized Alexander polynomial of a satellite knot with the one of its components (pattern and companion). Then, the result will be specified for satellites where the patterns are lassos.

\begin{defn}\label{framing}
A \emph{framed link} in an oriented 3-manifold $Y$ is a disjoint union of embedded circle, equipped with a non-zero normal vector field. Framed links are considered up to isotopy.
\end{defn}

\begin{defn}\label{sat-def}
Let $P$ be a knot in $ST$. We will call this knot \emph{pattern}. Let $C$ be a knot in $\mathbb{S}^3$ with framing $0$. We will call it \emph{companion}. Finally, let $e:ST\hookrightarrow\mathbb{S}^3$ be an embedding of $ST$ onto a regular neighbourhood of $C$, that maps a longitude of $ST$ onto a longitude of $C$. Then $eP$ is the \emph{satellite knot} of $P$ and $C$, hereinafter $Sat(P,C)$.
\end{defn}

\begin{nrmrk}
Considering $C$ to have framing $0$ is the same as taking it with blackboard framing ---the vector field lies on the plane where the diagram of the knot is represented--- and then adding $-wr(C)$ full-twists to the parallel knot that arises after we embed $P$ into a neighbourhood of $C$. Here $wr(C)$ means the writhe of $C$. For this construction we say that $P$ is sent \emph{faithfully}.
\end{nrmrk}

\begin{thm}[\cite{Li}]\label{lick-thm} Let $P$ be a knot in $ST$, $C$ be a knot in $\mathbb{S}^3$ and $Sat(P,C)$ be their satellite knot. Then,
\[\Delta_{Sat(P,C)}(t)=\Delta_P(t)\Delta_C(t^n),\]
where $P$ represents $n$ times a generator of $H_1(ST)$.
\end{thm}

We will use this result and prove the following statement.

\begin{prop}\label{alex}
Let $L\in \mathscr{L}_d$, $C$ be a knot in $\mathbb{S}^3$ and $Sat(L,C)$ be their satellite knot. Then,
\[\Delta_{Sat(L,C)}(t)=\Delta_C(t^d).\]
\end{prop}

\begin{proof}
We have two things to prove. The first one is that $\Delta_L(t)=1$, but this is trivial since every lasso is the trivial knot in $\mathbb{S}^3$, and so its Alexander polynomial is $1$. The remaining aspect to prove then is that $d=n$ in the previous theorem. How can we check this? Let us build a Seifert surface for $Sat(L,C)$ and then count how many times $L$ is a generator of $H_1(ST)$.

Consider the projection of $ST$ onto an annulus. Orient $L$ and apply the Seifert method \cite[Th. 2.2]{Li} to its normal diagram. We cap off with discs the circuits that remain bounded in the annulus; then we proceed to keep building the surface with the left circuits encircling the annulus. For our lasso case, $m+1$ Seifert circuits are generated encircling the annulus. When two adjacent circuits have opposite directions we pair them and use annuli to cap them off. From the number of unpaired circuits (which will be thought as copies of a generator of $H_1(ST)$) will arise $P$ as $n$ times a generator of $H_1(ST)$. It remains then to count properly this number $n$.

The Seifert circuits encircling the annulus will look like in the following picture (we will only depict one side of the composition).
\begin{align*}
&\includegraphics[scale=0.3]{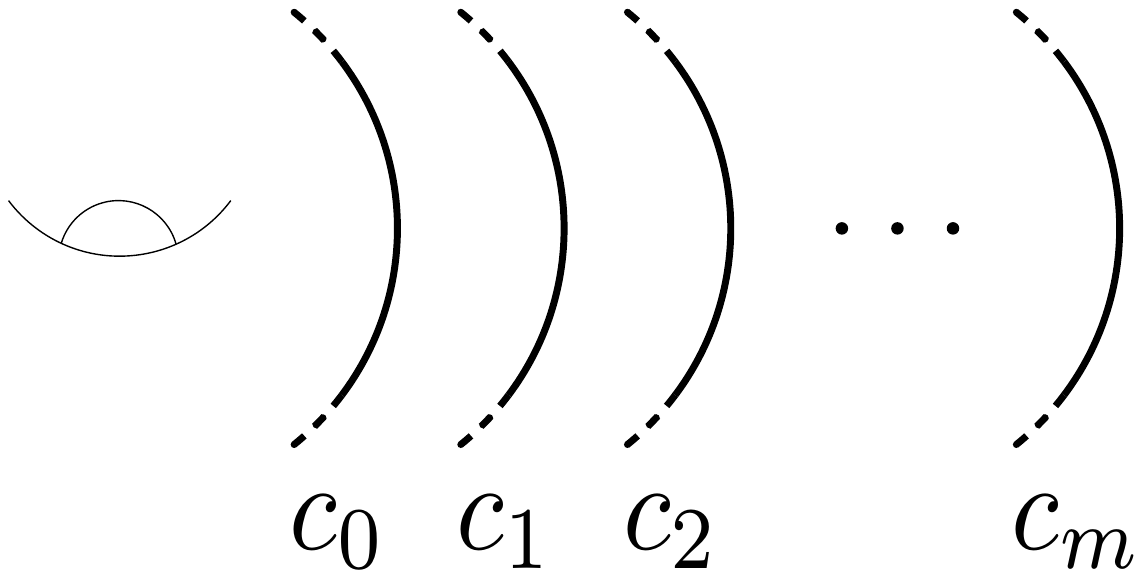}
\end{align*}
\indent We start supposing that all the circuits have the same direction.
\begin{align*}
&\includegraphics[scale=0.3]{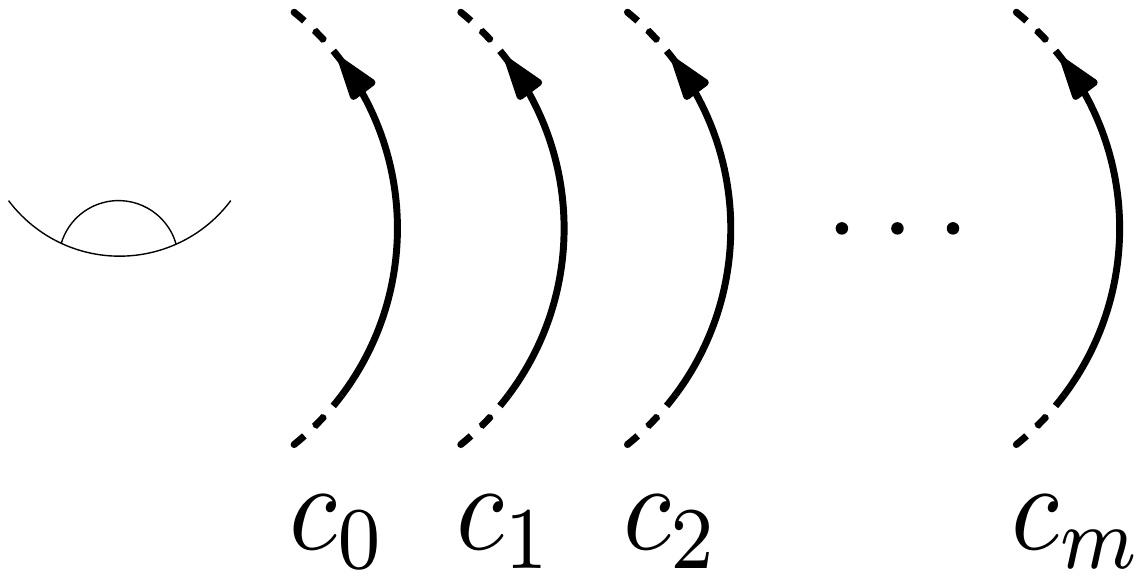}
\end{align*}
\indent That would result in $n=1+1+1+...+1=m+1$. Suppose now that one of the directions is reverted. We will start counting from the inner side of the projection.
\begin{align*}
&\includegraphics[scale=0.3]{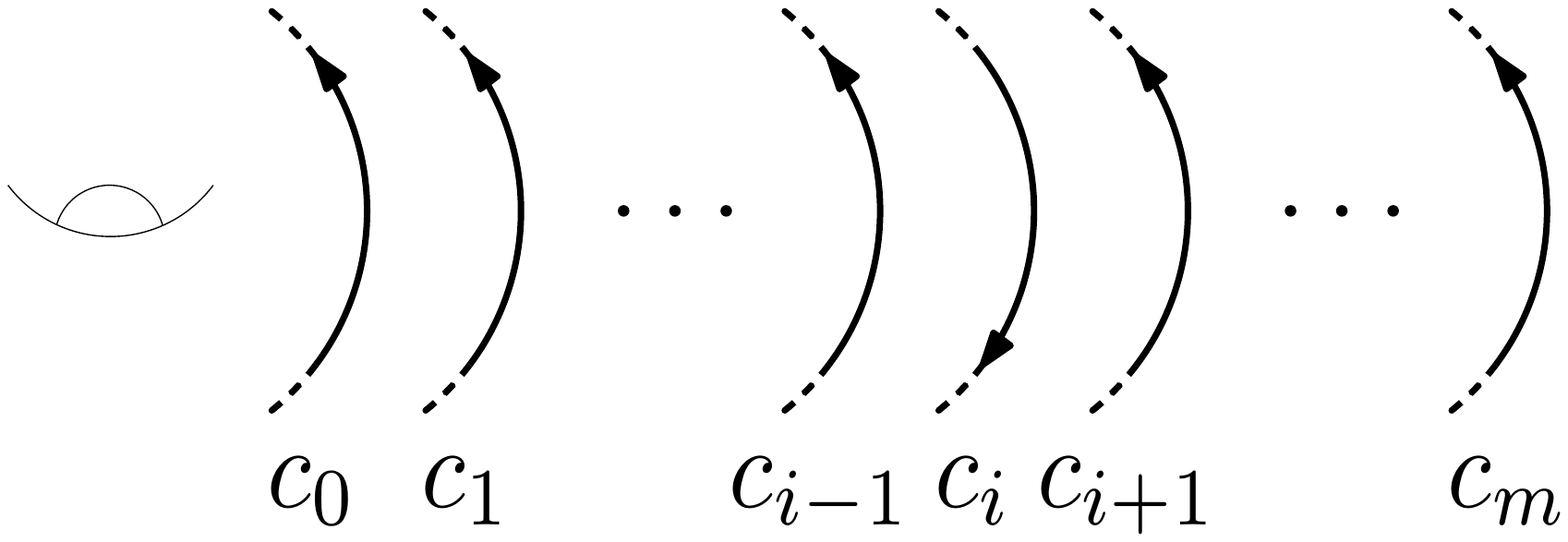}
\end{align*}
\indent In this case, $n=1+1+...+1-1+1+...+1=m+1-2=m-1$. It is important to notice that no matter what the orientation of $c_{i+1}$ is, it will always count positively to this matter, since it cannot be matched unless $c_{i+2}$ has opposite direction to it. Therefore, as one builds the Seifert surface, the number of times $L$ will represent a generator of $H_1(ST)$ (the unpaired circuits) is exactly the sum made after assigning to each circuit 1 if we start a ``new count'' after two paired circuits or the circuit has the ``same orientation'' as the previous one, and -1 if it has ``reverted orientation'' from the previous one. Thus, bearing in mind that the orientation depends on the number of twists of the strips in the original lasso (odd: same; even: revert), by \emph{Definition \ref{degree}} we obtain the desired equality.
\[n=\deg(L)=d\]
\end{proof}

\begin{nrmrk}
It has been pointed out by some authors (see \cite{Cr}) that $n$ also equals to the linking number of $P$ and a meridian of $ST$. Indeed, in our previous proof the count of Seifert circuits minding directions would stand for the count of the linking number of $P$ and a meridian of $ST$. Therefore, using this fact we could give an alternative proof to the previous one which would be the same in principle.
\end{nrmrk}

\section{Kauffman bracket skein module and Jones polynomial of satellite knots}

The purpose of this section is to give explicit formulae of the Kauffman bracket and the Jones polynomial of satellite knots making use of the Kauffman bracket skein module in $ST$.

\subsection{Kauffman bracket skein module}
Let us consider call $\mathcal{D}$ to the set of isotopy classes of diagrams of framed links in $ST$ projected onto the annulus.

\begin{defn}\label{kbsm}
We define the \emph{Kauffman bracket skein module} in $ST$ (and write $S(ST)$) as the free $\mathbb{C}[A^{\pm 1}]$-module with basis $\mathcal{D}$ over the smallest submodule containing the skein relation presented below.
\begin{align*}
&\includegraphics[scale=0.3, angle=90]{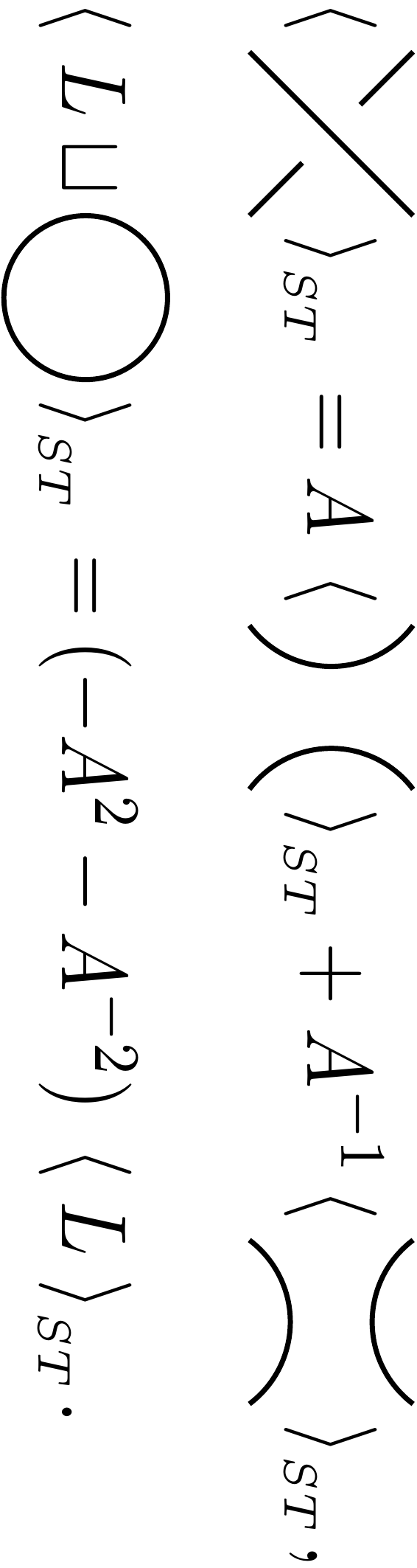}
\end{align*}
\end{defn}

In our case, if we consider as basis of $S(ST)$ the set $\mathcal{B^*}=\{z^{i*}\subST\;|\;i\geq 0\}$, where $z\subST$ is the core of $ST$, $S(ST)$ has an algebra structure that is the polynomial algebra $\mathbb{C}[A^{\pm 1};z\subST]$ (see \cite{Le}). We could depict this basis for $S(ST)$ as follows.
\begin{align*}
&\includegraphics[scale=0.3]{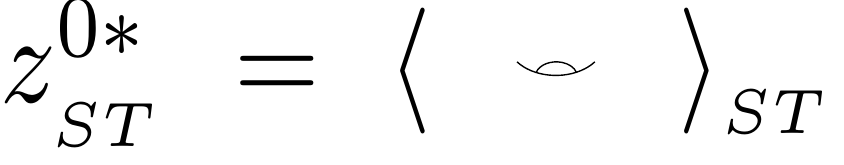} &&\spa{5}\:\includegraphics[scale=0.3]{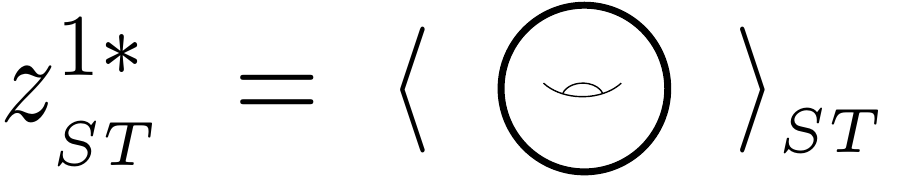} &&\raisebox{-0.4 em}{\includegraphics[scale=0.3]{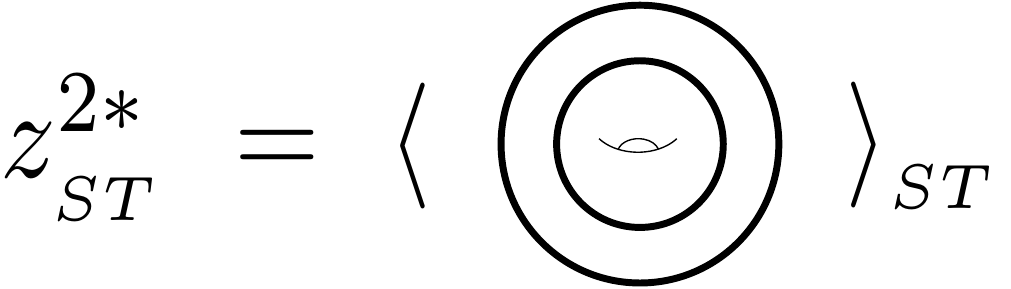}}\\
&& &\includegraphics[scale=0.3]{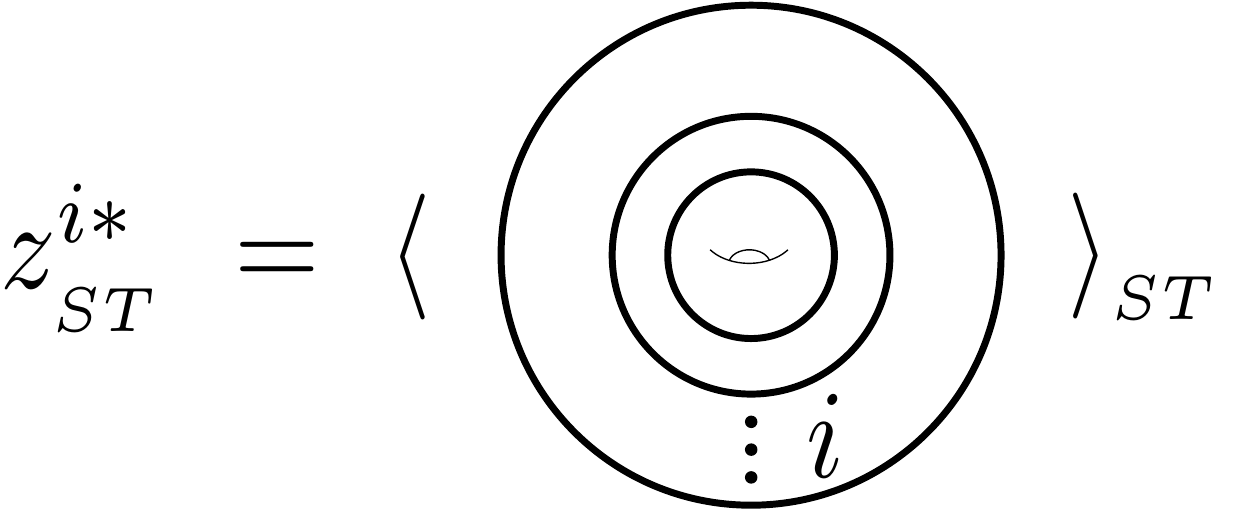}&&
\end{align*}

However, in practice we will consider a \textit{normalized} Kauffman bracket skein module ---for practical reasons---, which would be the previous where $z^0\subST=(-A^2-A^{-2})z^{0*}\subST$, so that in the subsequent part of the paper it matches with the Jones polynomial. Thus let us consider the following basis $\mathcal{B}$.
\begin{align*}
&z^0\subST=(-A^2-A^{-2})z^{0*}\subST &&z^i\subST=z^{i*}\subST
\end{align*}
\[\mathcal{B}=\{z^{i}\subST\;|\;i\geq 0\}.\]

This will be specially delicate in the later part where it will affect the product of basis elements in the following manner.
\begin{align*}
&z^0\subST z^j\subST=(-A^2-A^{-2})z^j\subST, &&z^i\subST z^j\subST=z^{i+j}\subST \qquad\textnormal{if $i,j\neq 0.$}
\end{align*}

When calculating, all knots will be taken with blackboard framing.

\begin{exmp}
Kauffman bracket in $ST$ of the left-handed trefoil regarded as the $(2,3)$-torus knot.
\begin{align*}
&\includegraphics[scale=0.3, angle=90]{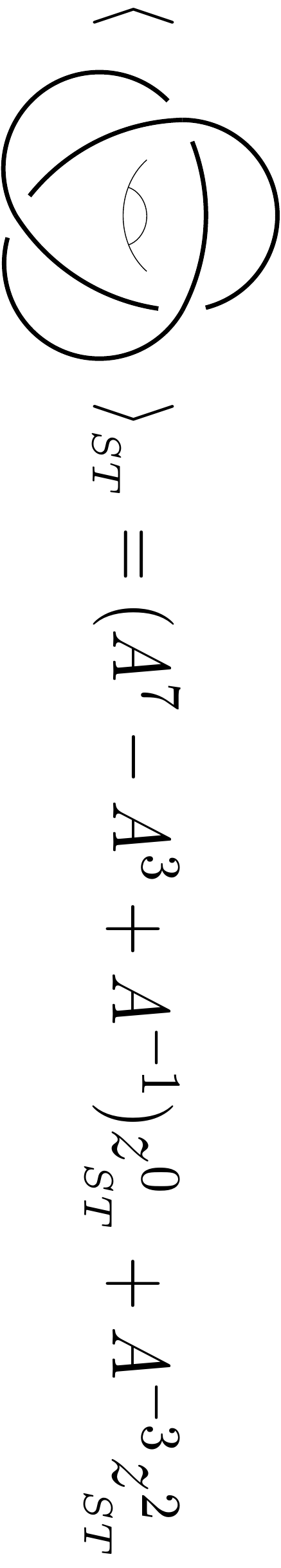}
\end{align*}
\end{exmp}
\begin{exmp}
Kauffman bracket in $ST$ of the embedded left-handed trefoil.
\begin{align*}
&\includegraphics[scale=0.3, angle=90]{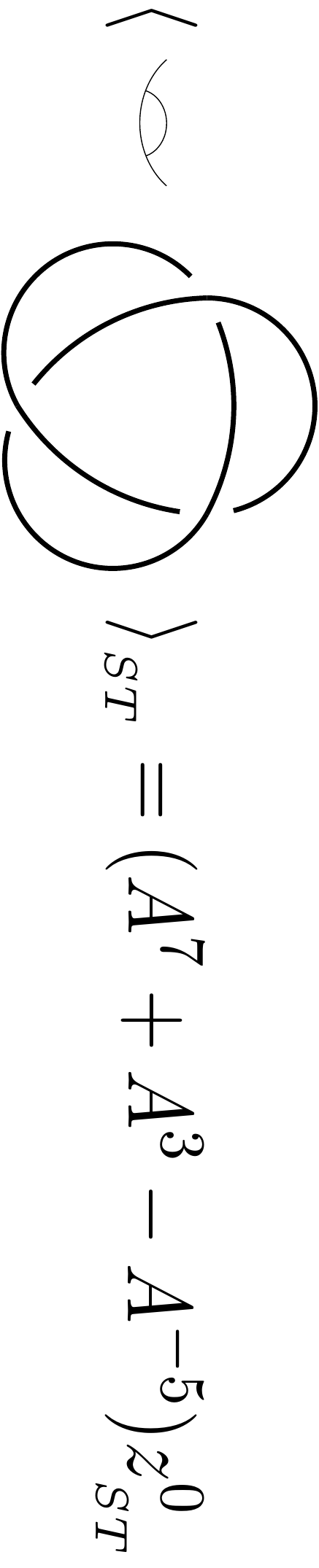}
\end{align*}

Observe that the result is coherent with the one that we obtain if we operated the usual Kauffman bracket of $3_1$ in $\mathbb{S}^3$.
\end{exmp}

Let us consider now $\mathscr{L}$ (the set of lassos) the object of study.

\begin{prop}\label{lasso-prop}
Let $L(r_1,r_2,...,r_m)\in \mathscr{L}$ be a lasso with normal diagram. Then its Kauffman bracket skein module in $ST$ can be calculated as follows.
\begin{align*}
&1.\;\kauf{L(\emptyset)}=\;z^1\subST;\spa{20}\spa{20}\spa{20}\\
&2.\;\kauf{L(0)}=\;z^0\subST;\\
&3.\;\kauf{L(0,r_2,...,r_m)}=\;T(r_2)\kauf{L(r_3,r_4,...,r_m)};\\
&4.\;\kauf{L(r_1,r_2,...,r_m)}= \left\{ 
\begin{array}{ll}
	A\kauf{L(r_1-1,r_2,...,r_m)}+A^{-1}z^1\subST T(r_1-1)\kauf{L(r_2,r_3,...,r_m)}\quad \textnormal{if $r_1>0$}\\
	A^{-1} \kauf{L(r_1+1,r_2,...,r_m)}+Az^1\subST T(r_1+1)\kauf{L(r_2,r_3,...,r_m)}\quad \textnormal{if $r_1<0$}
\end{array}
\right.
\end{align*}
where $T(n)=(-A^{-3})^n$.
\end{prop}

\begin{proof}
The first two equalities arise from their definition.
\begin{align*}
&\includegraphics[scale=0.3, angle=90]{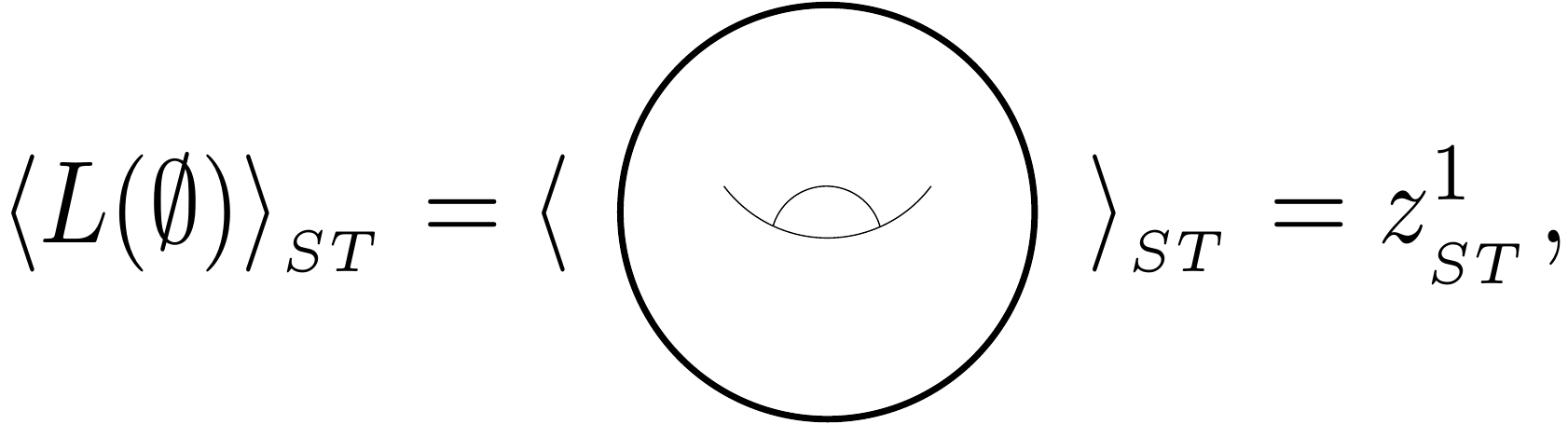}&&\raisebox{0.8 em}{\includegraphics[scale=0.3, angle=90]{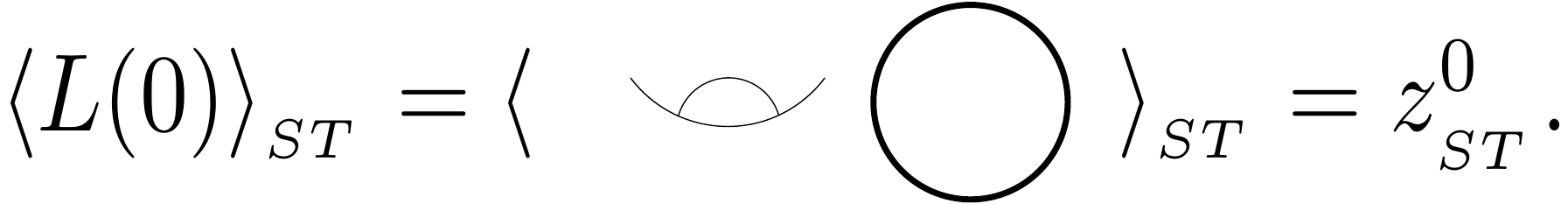}}
\end{align*}

The third one arises from the isotopy equivalence $L(0,r_2,r_3,...,r_m)\simeq L(r_3,...,r_m)$ that we saw earlier, where this time the twists unmade for that purpose count as a constant in the skein module.
\[\includegraphics[scale=0.45, angle=90]{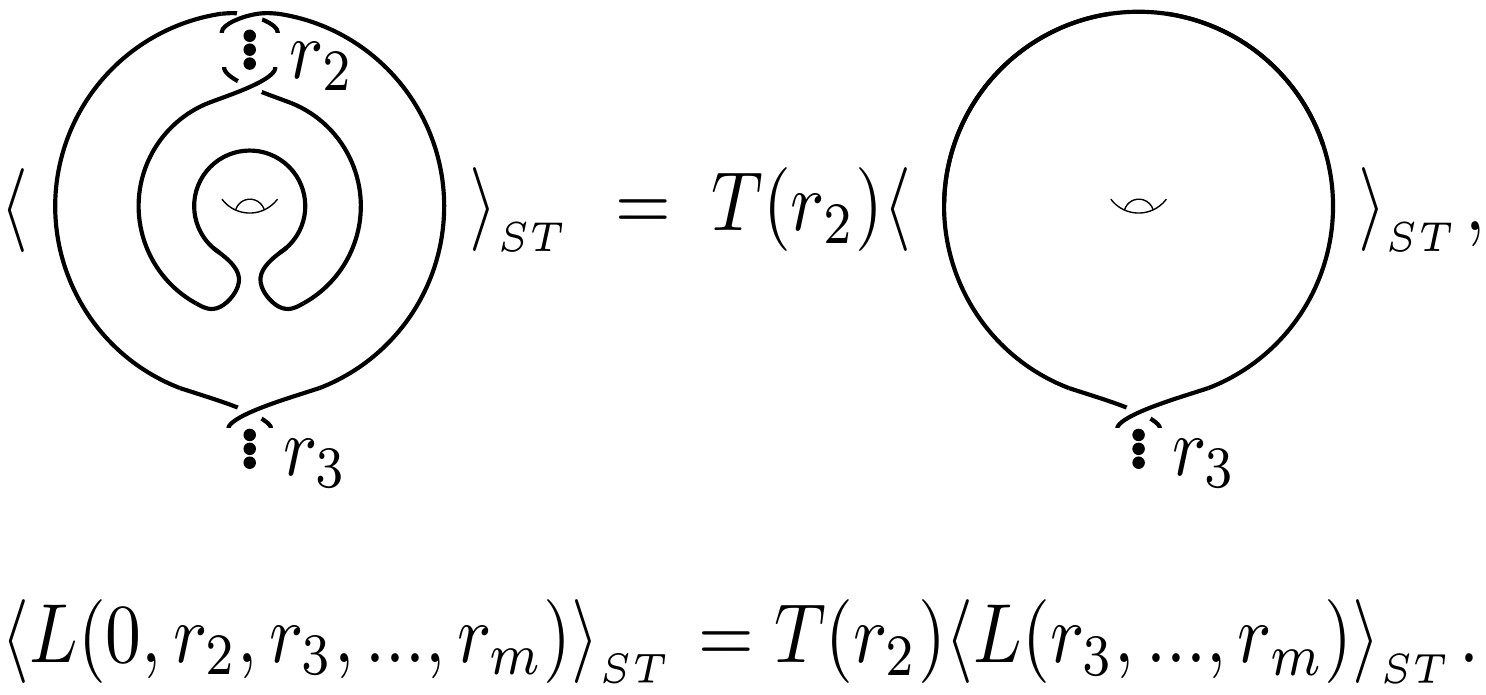}\]

Finally, we will depict the last equality considering $r_1 > 0$. If $r_1<0$ the proof is identical but for some signs.
It is enough, then, to apply the skein relation to the first crossing of the lasso counting from the innermost part to get the desired result. As we can see, the number of twists in the first diagram is reduced by $1$, whereas in the second diagram it is reduced ---and leads to a constant as in the above case--- and split into two disjoint diagrams ---what gives rise to the element $z^1\subST$.
\begin{align*}
&\includegraphics[scale=0.3, angle=90]{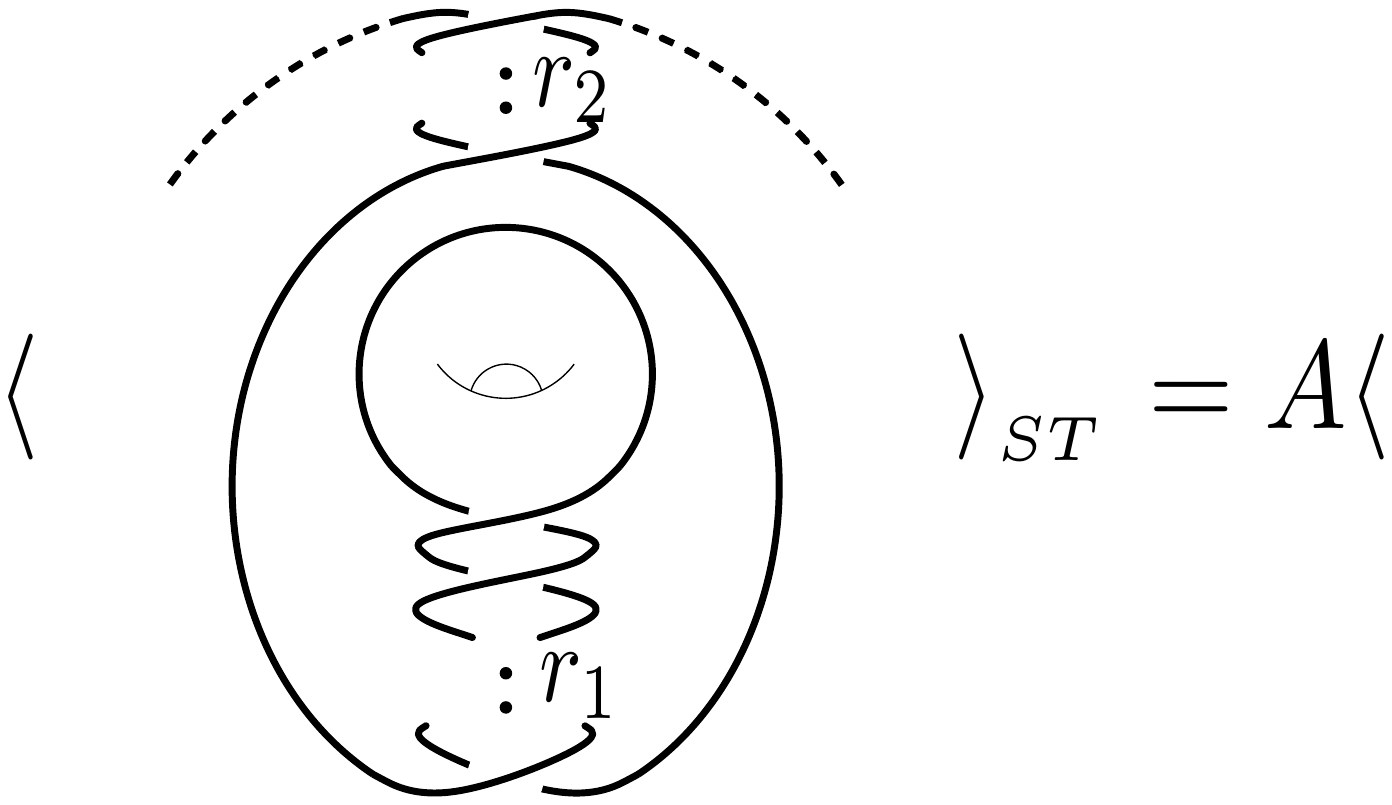}\quad\includegraphics[scale=0.3, angle=90]{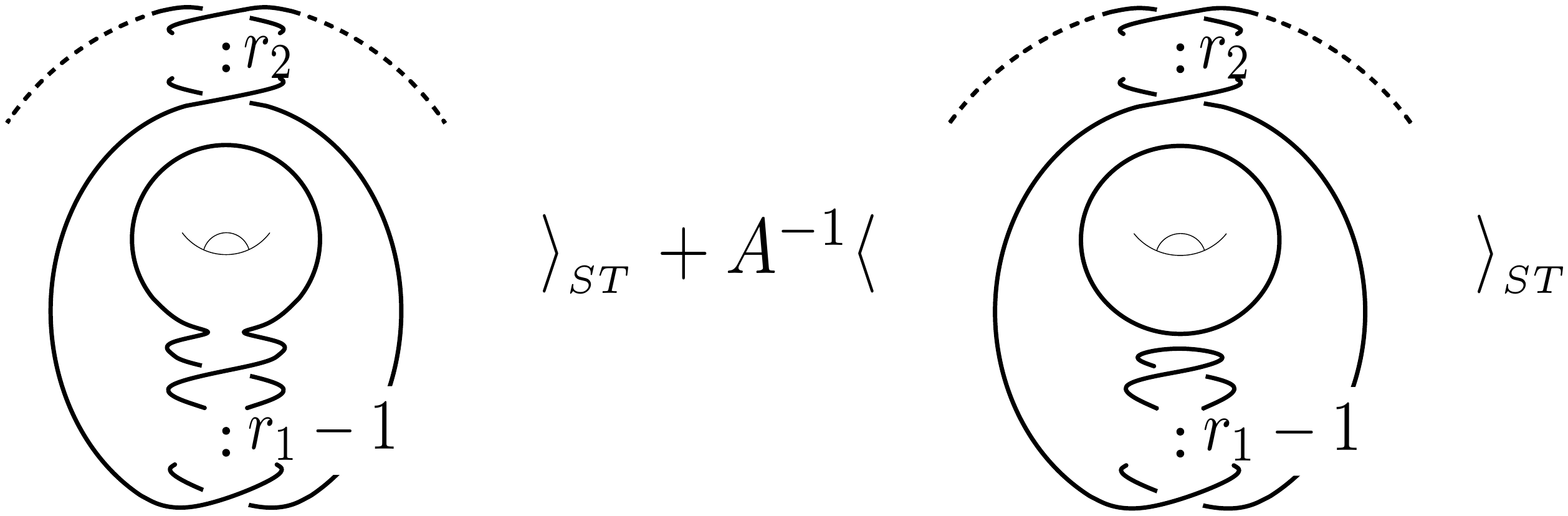}
\end{align*}
\[\kauf{L(r_1,r_2,...,r_m)}=A\kauf{L(r_1-1,r_2,...,r_m)}+A^{-1}z^1\subST T(r_1-1)\kauf{L(r_2,...,r_m)}.\]
\end{proof}

\begin{corol}\label{coro}
For the particular case of the simple lasso $L(r)$ with $r>0$, we obtain the next explicit formula.
\[\kauf{L(r)}=A^rz^0\subST+z^2\subST T(r)\sum\limits_{i=1}^r(-1)^iA^{4i-2}.\]
\end{corol}

\begin{proof}
We prove it inductively. When $r=1$ the result is true:
\[\kauf{L(1)}=Az^0\subST+A^{-1}z^2\subST.\]

Let us now suppose that the general case for $r$ is true and inspect the case $r+1$:
\begin{align*}
\kauf{L&(r+1)} = A\kauf{L(r)}+A^{-1}z^1\subST T(r)\kauf{L(\emptyset)}\\
&\spa{-1}\stackrel{hyp}{=}A\bigg[A^rz^0\subST+z^2\subST T(r)\sum\limits_{i=1}^r(-1)^iA^{4i-2}\bigg]+A^{-1}z^2\subST T(r)\\
&=A^{r+1}z^0\subST+z^2\subST T(r)\bigg[A\sum\limits_{i=1}^r(-1)^iA^{4i-2}+A^{-1}\bigg]=A^{r+1}z^0\subST-z^2\subST T(r+1)\bigg[A^4\sum\limits_{i=1}^r(-1)^iA^{4i-2}+A^2\bigg]\\
&=A^{r+1}z^0\subST+z^2\subST T(r+1)\sum^{r}_{i=0}(-1)^{i+1}A^{4(i+1)-2}=A^{r+1}z^0\subST+z^2\subST T(r+1)\sum^{r+1}_{i=1}(-1)^iA^{4i-2}.
\end{align*}

If $r<0$, then it suffices to replace every $A$ with $A^{-1}$ and the analogous result would be obtained.
\end{proof}

\subsection{Jones polynomial}

We will now make use of the previously defined Kauffman bracket skein module to express the Jones polynomial of a link in $ST$.

\begin{defn}\label{jones-st}
Let $L$ be an oriented framed link in $ST$ with a diagram $D$. We define the \emph{Jones polynomial} of $L$ in $ST$ (as resemblance of the Jones polynomial in $\mathbb{S}^3$) as the following.
\[J\subST(L)=T(wr(D))\kauf{D}\Big|_{t^{^1/_2}=A^{-2}},\]
where $T(n)=(-A^{-3})^n$ and $wr(D)$ is the writhe of the link diagram.
\end{defn}

The subindex $ST$ will be only specified for the solid torus case. For the $\mathbb{S}^3$ case (the usual case), no subindex will be written neither for the Jones polynomial nor for the Kauffman bracket.\\

In the case of $\mathbb{S}^3$ the only basis element of the bracket is $z^0\subS=\langle$ \raisebox{-0.4 em}{\includegraphics[scale=0.45]{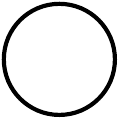}} $\rangle_{_{\mathbb{S}^3}}$ and in $ST$ the basis considered is $\mathcal{B}=\{z^i\subST\;|\;i\geq 0\}$, that is, $z^0\subST=\langle$ \raisebox{-0.4 em}{\includegraphics[scale=0.45]{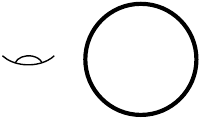}} $\rangle\subST$, $z^1\subST=\langle$ \raisebox{-0.4 em}{\includegraphics[scale=0.45]{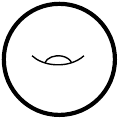}} $\rangle\subST$, $z^2\subST=\langle$ \raisebox{-0.75 em}{\includegraphics[scale=0.45]{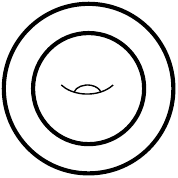}} $\rangle\subST$, etcetera.

\begin{nrmrk}
The choice of $z^0\subST$ was meant for this purpose, so that when a link $L$ in $ST$ is \underline{not} knotted around the center of $ST$ then its polynomial would coincide (except for the basis element of the skein module) with the polynomial calculated in $\mathbb{S}^3$. \[J\subST(\;\raisebox{0.2 em}{\includegraphics[scale=0.2]{Hole}}\;L\;)\big|_{z^0\subST=z^0\subS}=J(L).\]
\end{nrmrk}

\subsection{Working on satellites}

In this subsection we will consistently and constantly utilize the following notation:

\begin{itemize}
\item $L$ : (oriented) framed link in $\mathbb{S}^3$. 
\item $L^k$ : $k$-parallel of $L$. 
\item $L^k(n)$ : $k$-parallel of $L$ with $n\in\mathbb{Z}$ right-handed half-twists.
\item $U$ : trivial knot (unknot) in $\mathbb{S}^3$.
\item $M_P$ : \emph{geometric degree} of a knot $P$ in $ST$ (see \cite{St}; it can be also regarded as the maximum power of $z\subST$ in $\kauf{P}=\sum\limits_{i=0}^{M_P}a_iz^i\subST$). When this geometric degree appears in the construction of a satellite knot we will assume that it refers to the pattern knot and will only write $M$.
\end{itemize}

\begin{nrmrk}
Using this notation, one can regard $U^k(n)$ as the ($k$,$n$)-torus knot.
\end{nrmrk}

\begin{defn}\label{hot}
Let $D$ be a diagram of a knot $K$ in $ST$ projected onto an annulus. Consider a radius of the annulus that cuts $M_K$ arcs and no crossings of $D$, and take its complement in the annulus. We call a \emph{hot zone} of $D$ to the resulting region $H\simeq(\mathbb{S}^1\times I)\backslash (\{0\}\times I)\simeq(0,1)\times I$ in the annulus. When $K$ is considered in $\mathbb{S}^3$, we will call a \emph{hot zone} of $D$ to any rectangular region that contains all the crossings of the diagram. Typically, only one arc will be depicted out of the zone.
\end{defn}

By definition, a hot zone of a diagram contains all its crossings ---hence it also keeps all the information of the writhe. When abbreviated, we will represent it as a box with the label of the knot whose crossings contains, with $H$ as subindex.
\begin{align*}
& \includegraphics[scale=0.45, angle=90]{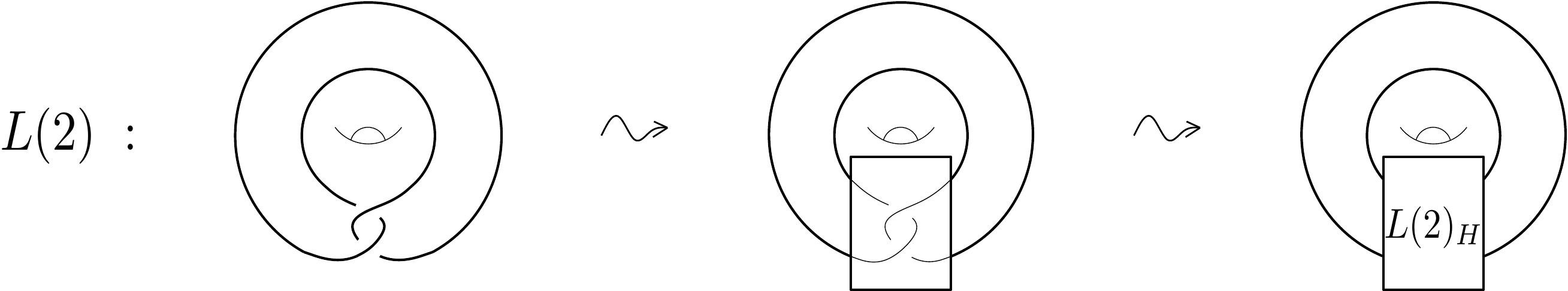}
\end{align*}

Hereinafter we will use the same naming for knots and their diagrams when calculi on diagrams occur.

\begin{defn}\label{composite}
Let $P$ be a knot in $ST$, and $C$ be a knot in $\mathbb{S}^3$ with blackboard framing. Let $M$ be the geometric degree of $P$. We can depict $P$ using a hot zone that encompasses its crossings as follows.
\begin{align*}
& \includegraphics[scale=0.30]{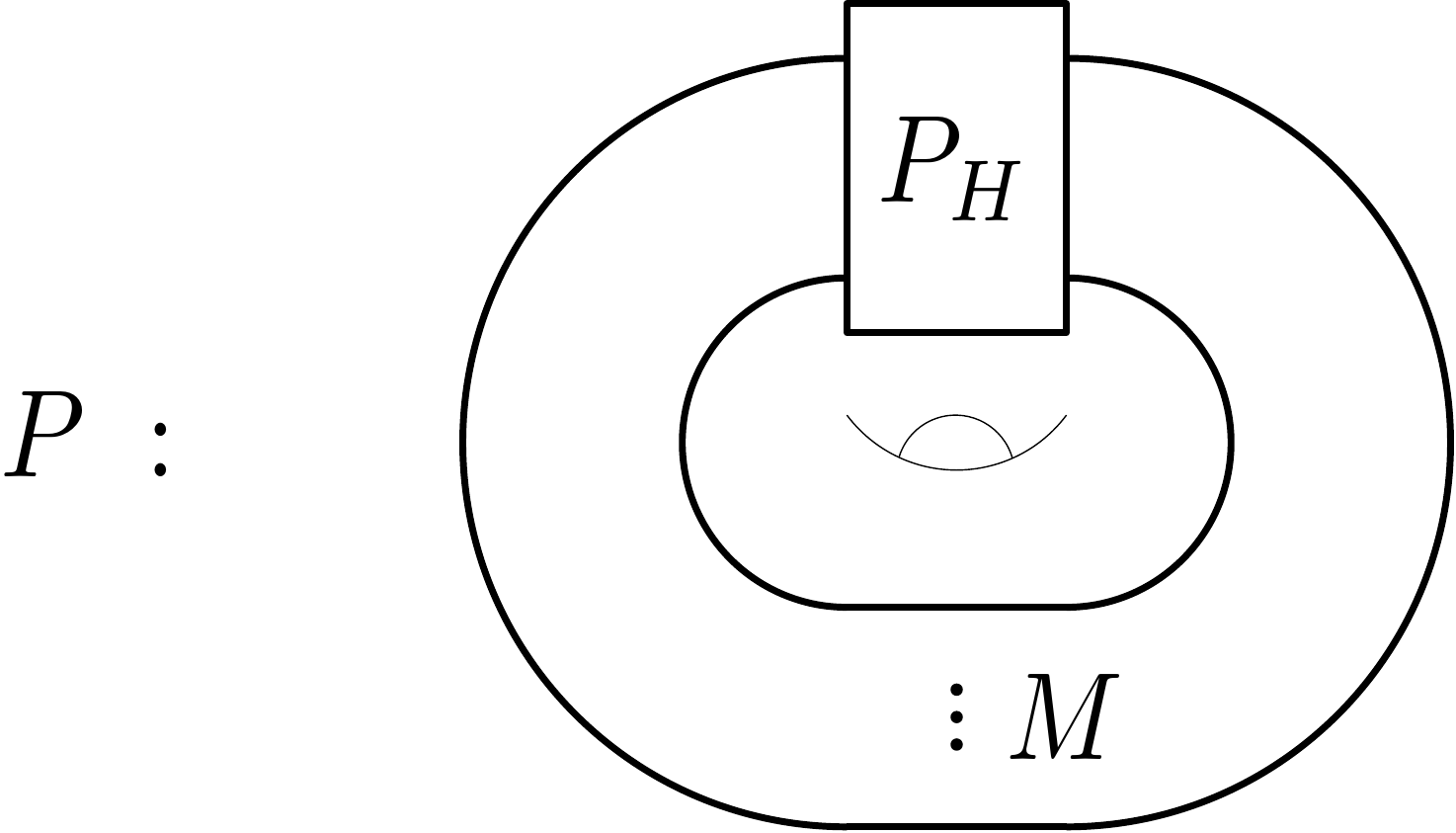}
\end{align*}
Here, $M$ is the number of arcs coming out from the hot zone $P_H$. Let now be $Sat(P,C)$ the satellite knot of $P$ and $C$. We will construct a diagram for the satellite from the diagrams of $P$ and $C$ orderly. By the definition of satellite knot, $P$ is sent into a neighbourhood of $C$, which we will portray as a hot zone $C_H$ with an only arc coming out of it. 
The result of the composition is considering the $M$-parallel of $C$ and then attaching $P_H$ to it as shown below.
\[\includegraphics[scale=0.45, angle=90]{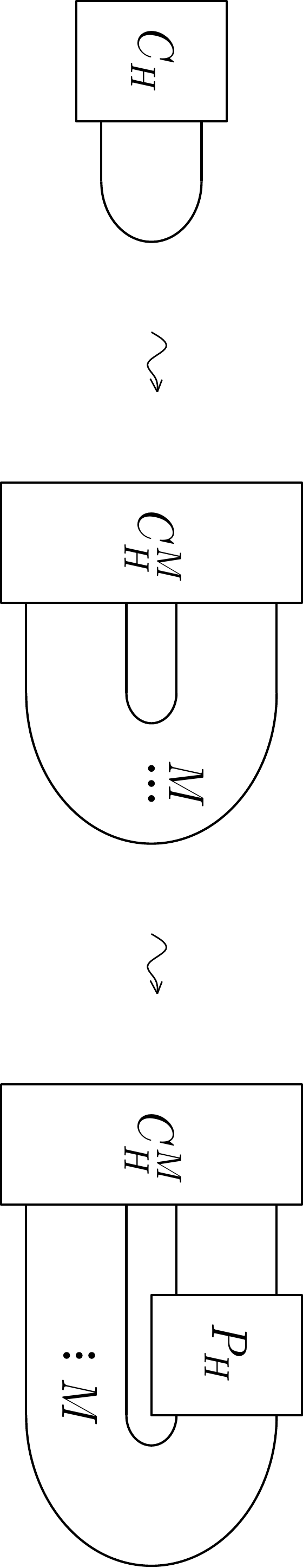}\]
In order for $P$ to be sent faithfully into a neighbourhood of $C$ we need to add as many full-twists to the construction as to compensate the writhe of $C$. The needed amount of full-twists is $-wr(C)$, that is, $-2wr(C)$ half-twists. We represent these twists using the given notation for the unknot, to which we add $-2wr(C)$ half-twists. Therefore, the needed hot zone in our diagram is $U^M_H(-2wr(C))$. The result will be the final form of the diagram for our satellite, and it is what we will call the \emph{composite diagram} of $Sat(P,C)$. It would look like the picture below.
\[\includegraphics[scale=0.30]{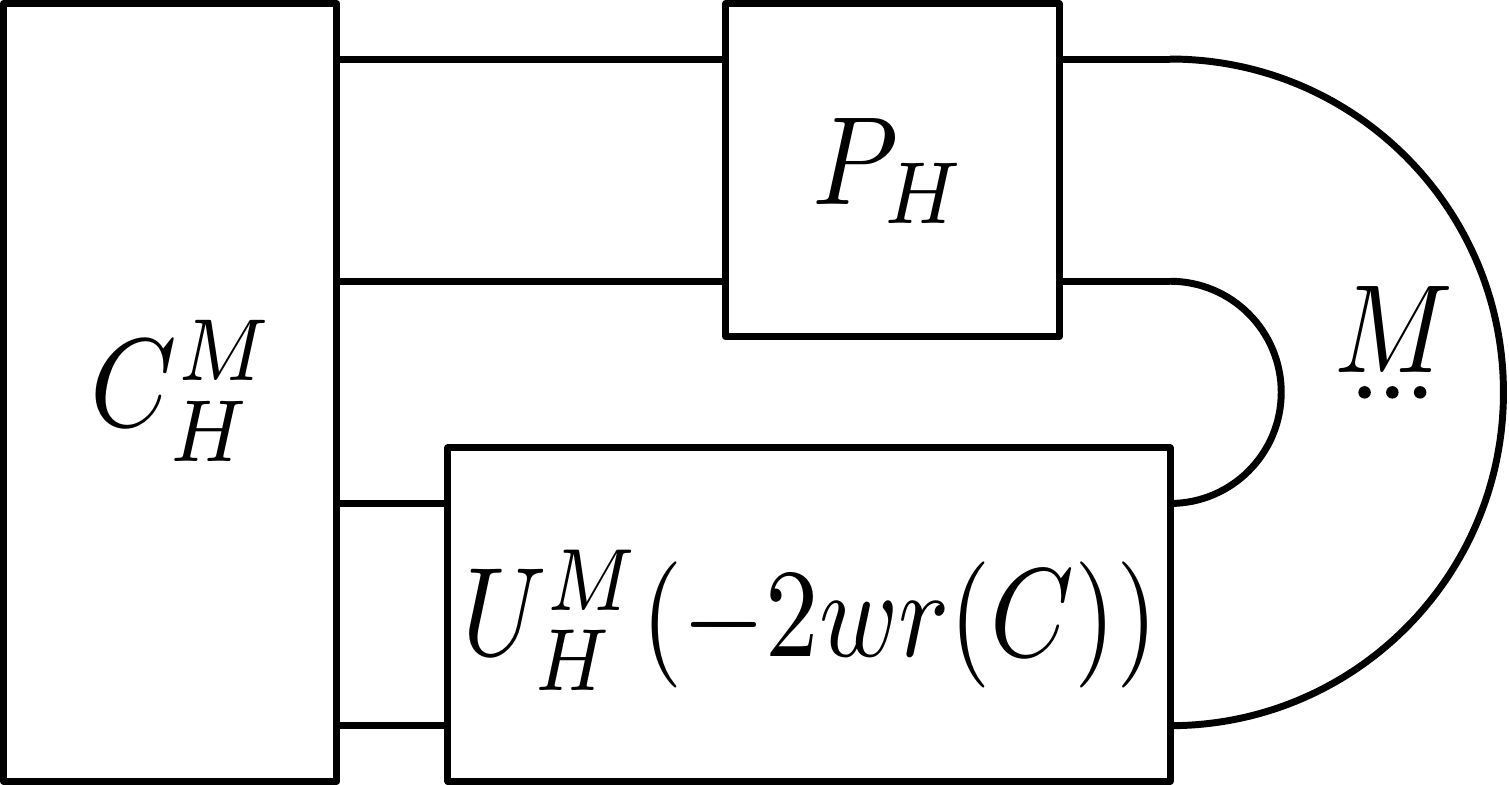}\]
\end{defn}

\begin{thm}\label{kauf}
Let $P$ be a knot in $ST$, $C$ be a knot in $\mathbb{S}^3$ and $Sat(P,C)$ be their satellite knot with composite diagram. Then, the following equality holds.
\[\kaufS{Sat(P,C)}=\kauf{P}\Big|_{z^k\subST=T(-wr(C))^{M-k}\kaufS{C^k(-2wr(C))}},\]
where $T(n)=(-A^{-3})^n$, $wr(C)$ is the writhe of $C$ and $0\leq k\leq M$.
\end{thm}

\begin{proof}
As previously stated, we will represent by boxes the hot zones of knots. Remember that knots and their diagrams are equally named.
Given the diagrams of $P$ and $C$, we consider the composite diagram of $Sat(P,C)$ as previously represented.
\[\includegraphics[scale=0.3]{Prop2-4}\]
We will focus on resolving $P_H$ using the skein relations of the Kauffman bracket skein module ---remember that we are considering blackboard framing. When resolving $P_H$ in $P$ (in $ST$) we are left with a polynomial on the basis $\mathcal{B}=\{z^k\subST\;|\;0\leq k\leq M\}$. By repeating the exact same process in $Sat(P,C)$ (in $\mathbb{S}^3$) we will be left with the same polynomial on the basis $\mathcal{B}_{Sat}=\{\kaufS{C^k(-2wr(C))}\;|\;0\leq k\leq M\}$. More in detail, each $z^k\subST$ will be transformed into $\kaufS{C^k(-2wr(C))}$ in the following manner.
\begin{itemize}
\item The biggest basis element $z^M\subST\in\mathcal{B}$ is transformed into the element $\kaufS{C^M(-2wr(C))}\in\mathcal{B}_{Sat}$.
\[\includegraphics[scale=0.45, angle=90]{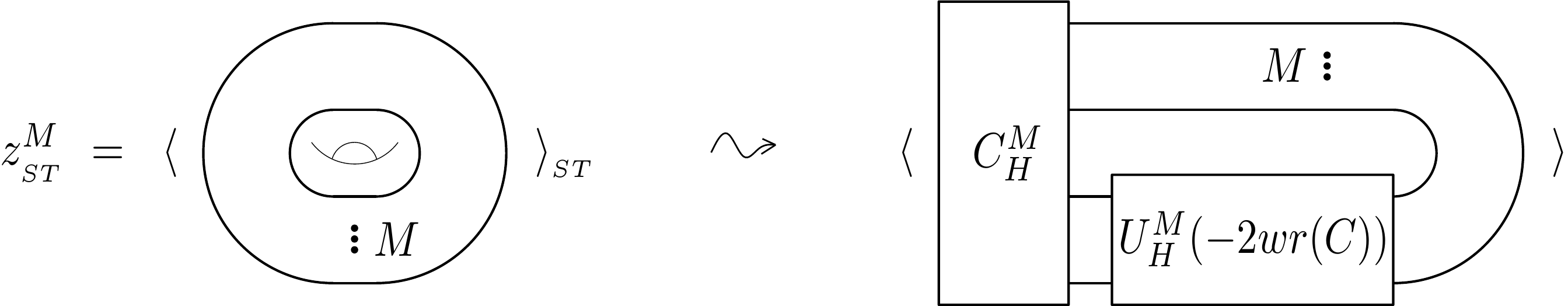}\]
\item From this point on, the transformations of the rest of the elements of the basis that appear carry a constant. The element $z^{M-2}\subST\in\mathcal{B}$ ---which is the next to appear--- is transformed into $T(-2wr(C))\kaufS{C^{M-2}(-2wr(C))}$, where the element of the basis $\kaufS{C^{M-2}(-2wr(C))}\in\mathcal{B}_{Sat}$ has been multiplied by $T(-2wr(C))$. This constant is turn obtained from a twisted unknot traversing $U^M_H(-2wr(C))$, when retracted to an untwisted unknot. The pictures below explain graphically the appearance of this constant.

\[\includegraphics[scale=0.45, angle=90]{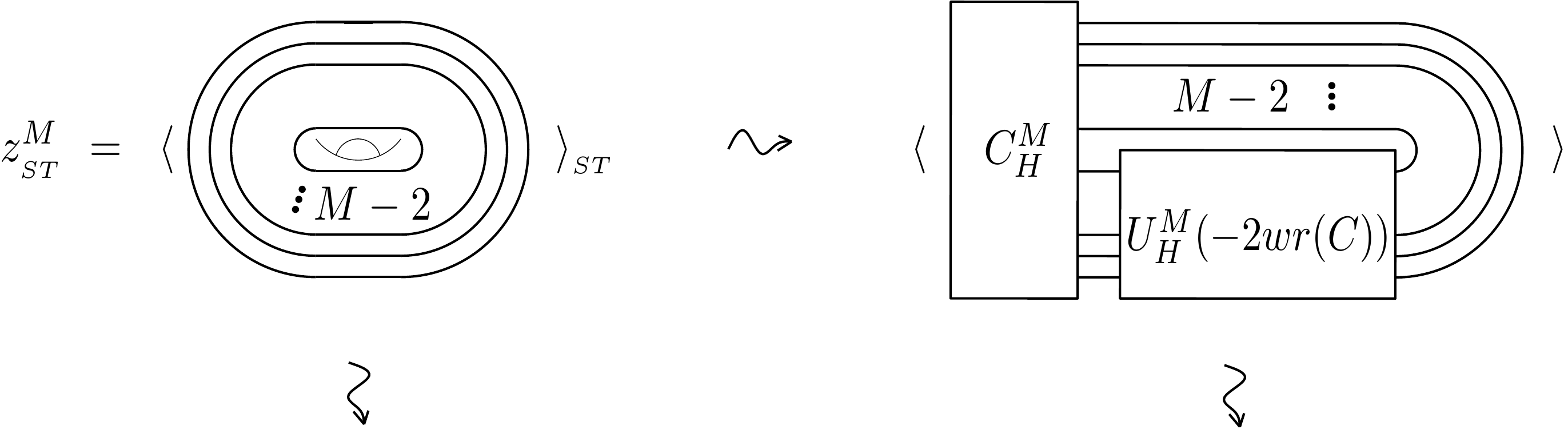}\]
\[\spa{15}\includegraphics[scale=0.45, angle=90]{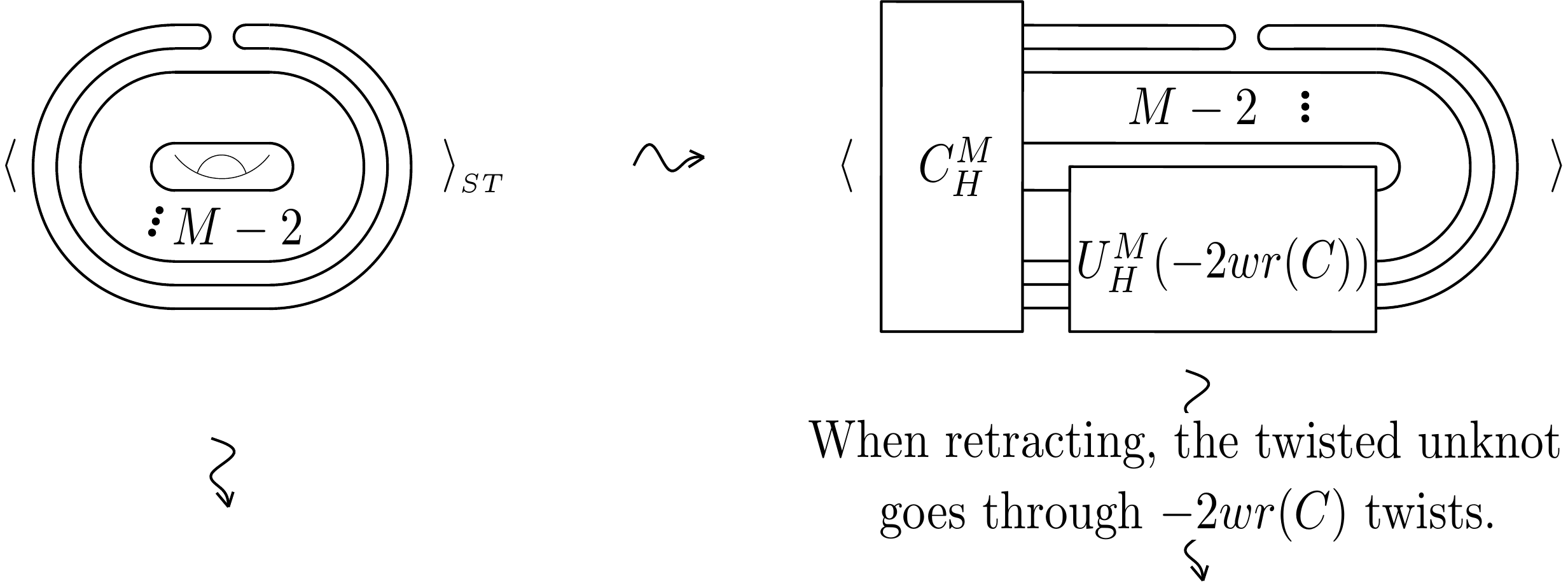}\]
\[\includegraphics[scale=0.45, angle=90]{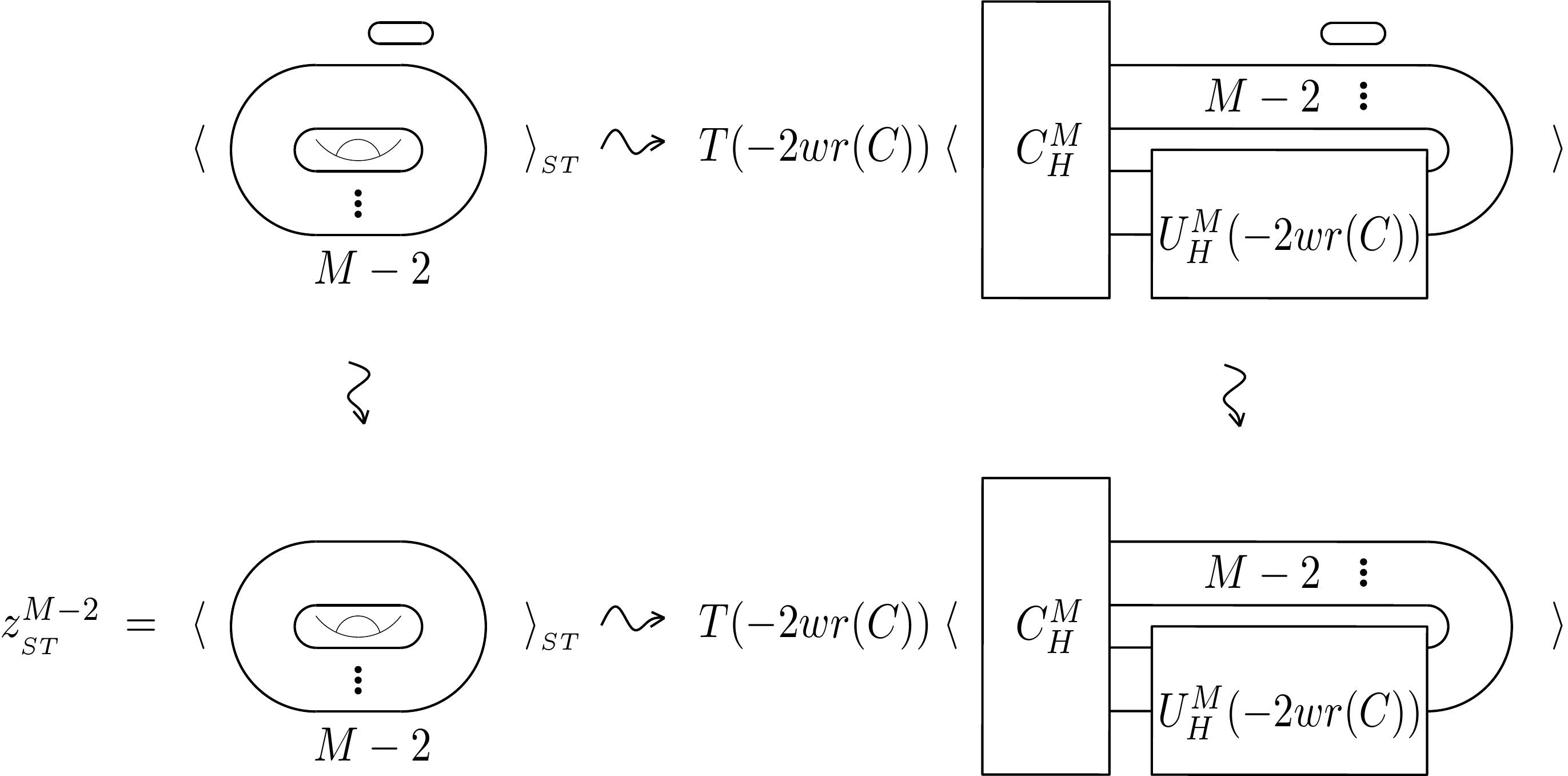}\]
\end{itemize}

As in previous occasions, the constant $T(n)$ equals to $(-A^{-3})^n$. Therefore the transformation is attained. Note that since all copies of $C$ regarded in $C^M(-2wr(C))$ are parallel, the result remains the same in spite of what two copies are united. If we now repeat the same reasoning consecutively, we find out that, in general, $z^k\subST$ is transformed into $T(-2wr(C))^{\frac{M-k}{2}}\kaufS{C^k(-2wr(C))}$, and hence we get the promised equality.
\end{proof}

\begin{exmp}
We calculate the Kauffman bracket the satellite knot that uses the left-handed trefoil regarded as the $(2,3)$-torus knot as pattern, and the eight-figure knot ---whose write is $0$--- as companion.
\begin{align*}
&\includegraphics[scale=0.3, angle=90]{Trefoil-kauf-1}\\
&\includegraphics[scale=0.3, angle=90]{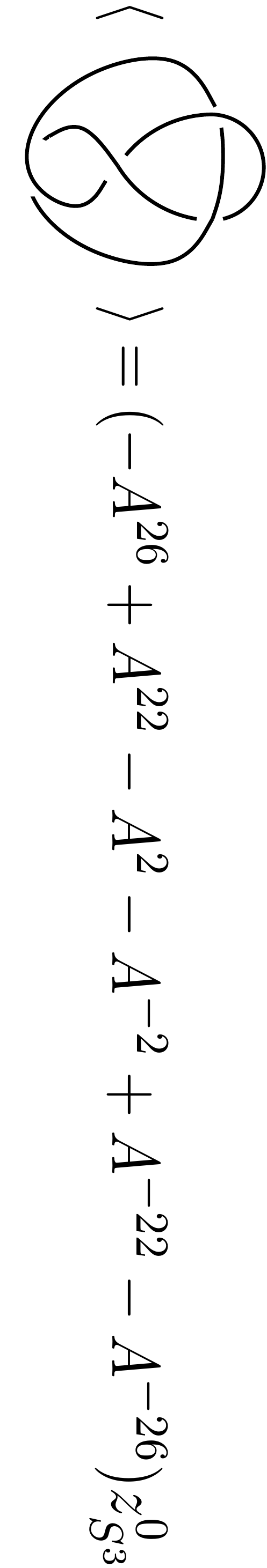}\\
&\raisebox{-1.9 em}{\includegraphics[scale=0.3, angle=90]{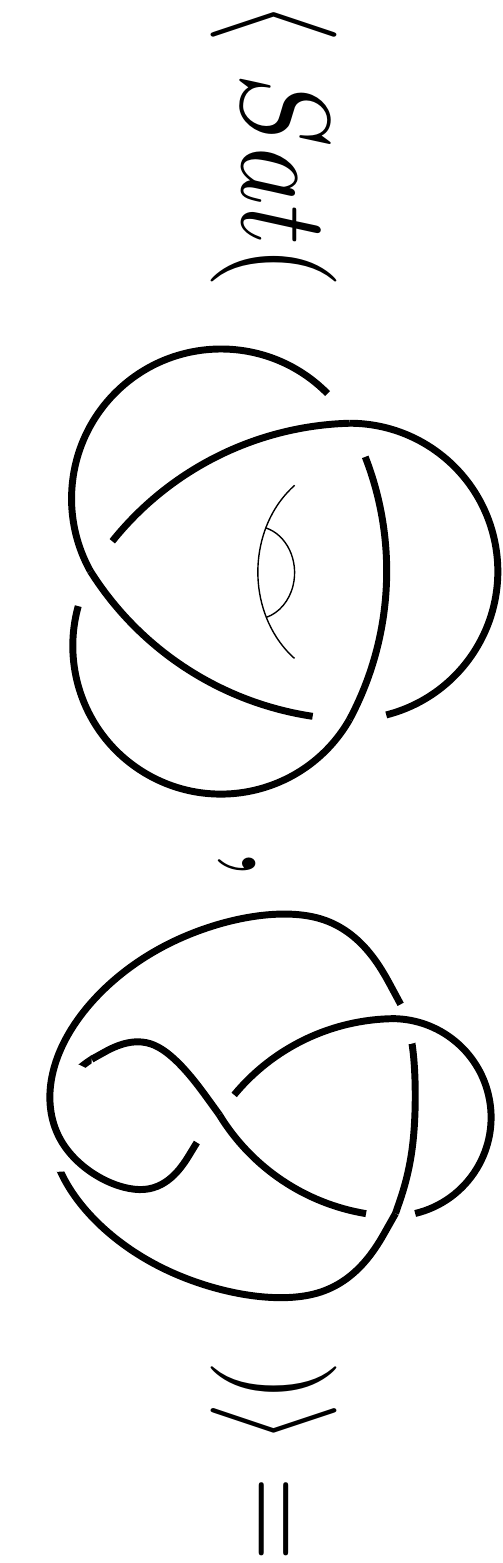}}(A^7-A^3+A^{-1})z^0\subST+A^{-3}z^2\subST\Big|_{\includegraphics[scale=0.2, angle=90]{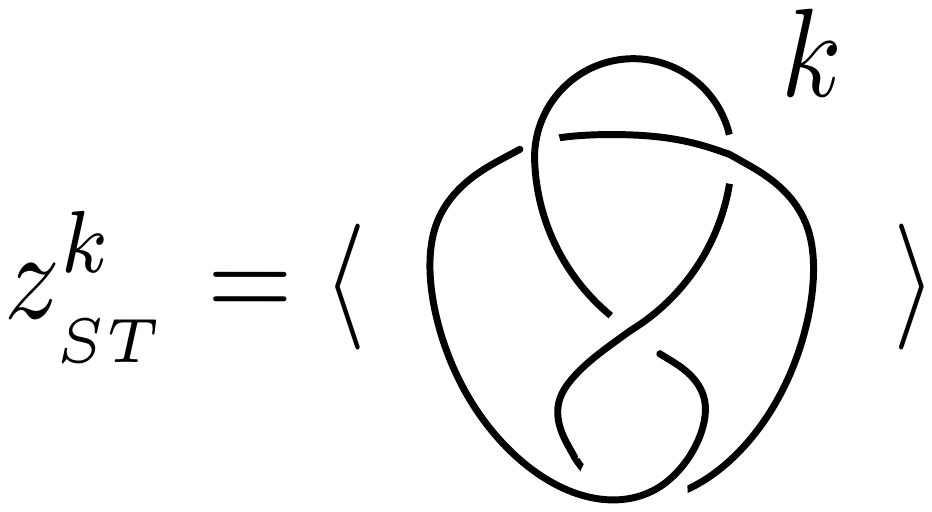}}\\
&\spa{14}=(A^7-A^3+A^{-1})z^0\subS+A^{-3}(-A^{26}+A^{22}-A^2-A^{-2}+A^{-22}-A^{-26})z^0\subS\\
&\spa{14}=(-A^{23}+A^{19}+A^7-A^3-A^{-5}+A^{-25}-A^{-29})z^0\subS
\end{align*}
\end{exmp}

\begin{lmm}\label{writhe}
Let $P$ be a knot in $ST$, $C$ be a knot in $\mathbb{S}^3$ and $Sat(P,C)$ be their satellite knot with composite diagram. Then, the writhe of the satellite knot can be expressed as follows.
\[wr(Sat(P,C))=wr(P)+Mwr(C).\]
\end{lmm}

\begin{proof}
A simple calculus of the writhe of an $n$ times half-twisted $k$-parallel knot $K$ shows us that
\[wr(K^k(n))=k^2wr(K)+n\frac{k(k+1)}{2}.\]
Glancing at the composite diagram of $Sat(P,C)$, we merely count:
\begin{align*}
wr(Sat(P,C))=&wr(P)+wr(C^M(-2wr(C)))=wr(P)+M^2wr(C)+(-2wr(C))\frac{M(M-1)}{2}\\
=&wr(P)+M^2wr(C)-M(M-1)wr(C)=wr(P)+Mwr(C).
\end{align*}
\end{proof}

The next theorem expresses the usual Jones polynomial of a satellite knot in terms of the Jones polynomial in $ST$ of its pattern and the parallel version of the Jones polynomial of its companion.

\begin{thm}\label{main}
Let $P$ be a knot in $ST$, $C$ be a knot in $\mathbb{S}^3$ and $Sat(P,C)$ be their satellite knot. Then, the following equality holds.
\[J(Sat(P,C))=J\subST(P)\Big|_{z^k\subST=J(C;\:k)},\]
where $0\leq k\leq M$, and $J(C;k)$ represents the $k$-parallel Jones polynomial.
\end{thm}

\begin{proof}
By the definition of the Jones polynomial through the Kauffman bracket, we can express the Jones polynomial of $Sat(P,C)$ as follows.
\[J(Sat(P,C))=T(wr(Sat(P,C)))\kaufS{Sat(P,C)}\Big|_{t^{^1/_2}=A^{-2}},\]
where $T(n)=(-A^{-3})^n$. Let us analyze the right side of the above equality leaving out the substitution $t^{^1/_2}=A^{-2}$. The following first two equalities make use of \emph{Lemma \ref{writhe}} and \emph{Theorem \ref{kauf}} respectively.
\begin{align*}
T(wr&(Sat(P,C)))\kaufS{Sat(P,C)}=\\
&=T(wr(P))T(Mwr(C))\kaufS{Sat(P,C)}\\
&=T(wr(P))T(Mwr(C))\kauf{P}\Big|_{z^k\subST=T(-wr(C))^{M-k}\kaufS{C^k(-2wr(C))}}\\
&=T(wr(P))T(wr(C))^M\kauf{P}\Big|_{z^k\subST=T(wr(C))^{k-M}\kaufS{C^k(-2wr(C))}}\\
&=T(wr(P))\kauf{P}\Big|_{z^k\subST=T(wr(C))^MT(wr(C))^{k-M}\kaufS{C^k(-2wr(C))}}\\
&=T(wr(P))\kauf{P}\Big|_{z^k\subST=T(wr(C))^k\kaufS{C^k(-2wr(C))}}.
\end{align*}

We now reintroduce the substitution $t^{^1/_2}=A^{-2}$ without performing it.
\begin{align*}
J(Sat&(P,C))=\\
=&\Big(T(wr(P))\kauf{P}\Big|_{z^k\subST=T(kwr(C))\kaufS{C^k(-2wr(C))}}\Big)\Bigg|_{t^{^1/_2}=A^{-2}}\\
=&\big(T(wr(P))\kauf{P}\big)\Big|_{t^{^1/_2}=A^{-2}}\Bigg|_{z^k\subST=\big(T(kwr(C))\kaufS{C^k(-2wr(C))}\big)\Big|_{t^{^1/_2}=A^{-2}}}\\
=&J\subST(P)\Big|_{z^k\subST=J(C^k(-2wr(C)))}\\
=&J\subST(P)\Big|_{z^k\subST=J(C;\:k)}.
\end{align*}

This last equality arises from the geometrical interpretation of the $k$-parallel Jones polynomial: considering $C$ with blackboard framing, taking its $k$-parallel and adding as many full-twists as the opposite of the writhe of $C$ ---that is, $J(C^k(-2wr(C)))$--- is the same as considering $C$ with framing 0 and taking its $k$-parallel ---that is, $J(C;\:k)$--- (see \cite{Mu}).
\end{proof}

\begin{nrmrk}
The above result also implies the already known formula for the Jones polynomial of the connected sum of two knots, since the knot sum is a special case of satelliting. If we consider $K_1$ in $ST$ with geometric degree $1$ and $K_2$ in $\mathbb{S}^3$, their satellite knot equals to the connected sum of both knots.
\[J(Sat(K_1,K_2))=J(K_1\#K_2)=J(K_1)J(K_2)\]
\end{nrmrk}

\begin{exmp}
We calculate the Jones polynomial of the satellite knot that uses the left-handed trefoil regarded as the $(2,3)$-torus knot as pattern, and the eight-figure knot as companion.
\begin{align*}
&\includegraphics[scale=0.3, angle=90]{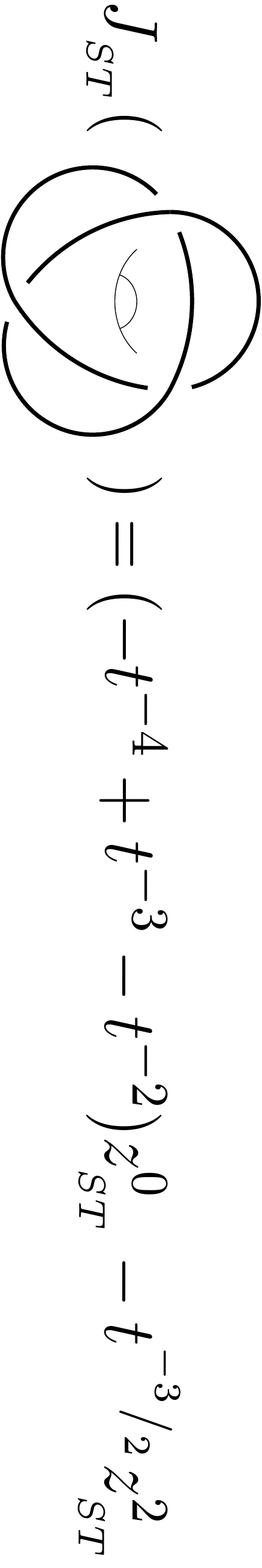}\\
&\raisebox{-1 em}{\includegraphics[scale=0.6, angle=90]{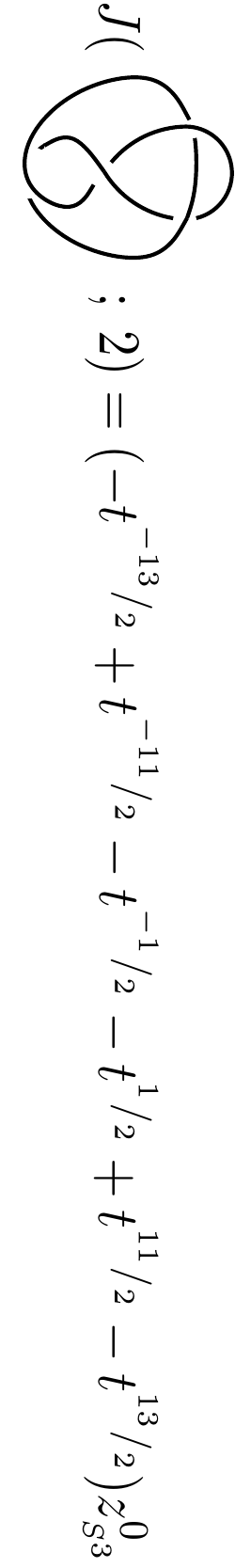}}\\
&\raisebox{-1.9 em}{\includegraphics[scale=0.3, angle=90]{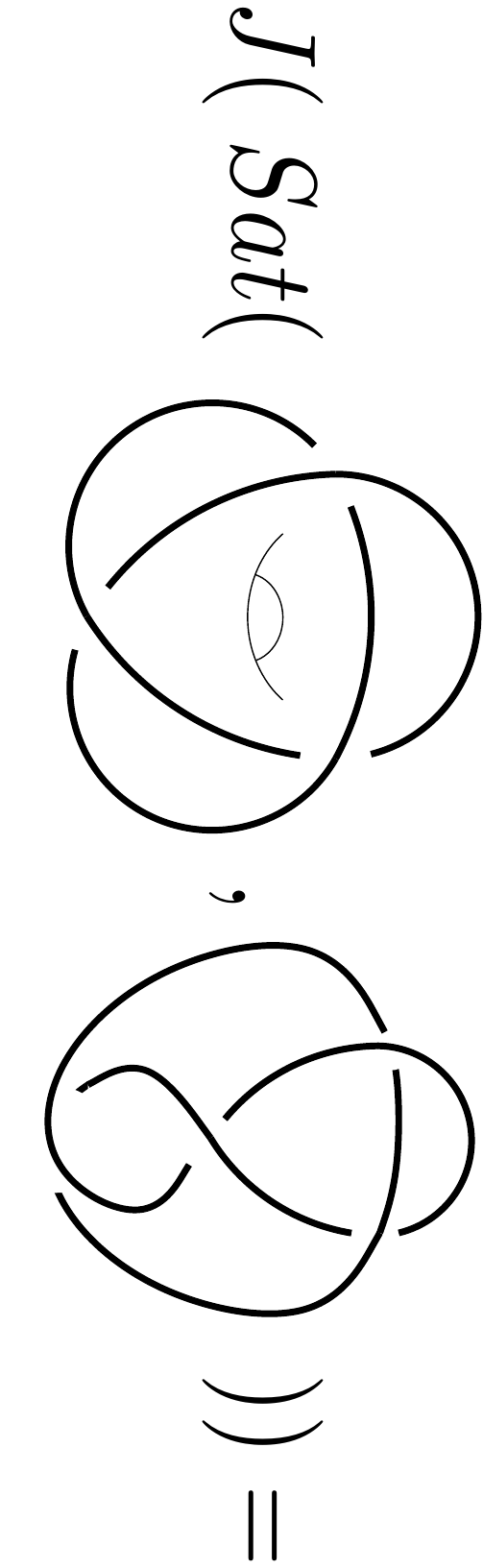}}\\
&\spa{14}=(-t^{-4}+t^{-3}-t^{-2})-t^{^{-3}/_2}(-t^{^{-13}/_2}+t^{^{-11}/_2}-t^{^{-1}/_2}-t^{^1/_2}+t^{^{11}/_2}-t^{^{13}/_2})z^0\subS\\
&\spa{14}=(t^{-8}-t^{-7}-t^{-4}+t^{-3}+t^{-1}-t^4+t^5)z^0\subS\\
&\spa{14}=t^{-8}-t^{-7}-t^{-4}+t^{-3}+t^{-1}-t^4+t^5.
\end{align*}
\end{exmp}

\newpage

\section{Infinitely many knots with the same Alexander polynomial and different Jones polynomial}

In this section we will make use of the results of the previous section, and apply them to some examples of the first section. The plan is easy: we will give some infinite families of knots (based on lassos) with the property that all the knots in each family share their Alexander polynomial, but they are not the same since the Jones polynomial tells them appart. Please bear in mind that all calculi in the following subsections will be done for $r>0$. If $r$ is considered to be negative, then it suffices to replace in the formulae for Kauffman brackets $A$ with $A^{-1}$, and for Jones polynomials $t$ with $t^{-1}$ (and viceversa). Let us begin.

\subsection{$L(r)$: lassos of degree 0 or 2}
First, we work with the previously defined Kauffman bracket skein module in $ST$, and we use \emph{Proposition \ref{lasso-prop}} to calculate the Kauffman bracket of the family $L(r)$ where $r>0$. The first equality that we show is the one proved to be true in \emph{Corollary \ref{coro}}.
\begin{align*}
&\kauf{L(r)}=A^rz^0\subST+T(r)\sum\limits_{i=1}^r(-1)^iA^{4i-2}z^2\subST,\\
&T(wr(L(r)))=T(-r).
\end{align*}
Here, $T(n)$ is defined as in the previous section as $T(n)=(-A^{-3})^n$. Then, applying the definition of the Jones polynomial in $ST$ (\emph{Definition \ref{jones-st}}) we get the general formula for simple lassos shown below.
\begin{equation}\label{L(r)}
J\subST(L(r))=(-t)^{-r}z^0\subST-t^{^1/_2}\frac{1-(-t)^{-r}}{t+1}z^2\subST.
\end{equation}

Let us suppose now that $|t|>1$. We will prove the next result.

\begin{thm}\label{lasso-thm1}
Let $L(r)$ be the family of simple lassos and let $C$ be a knot in $\mathbb{S}^3$ such that $\Delta_C(t)\neq 1$ and $J(C;2)\neq J(U;2)$. Then, $Sat(L(r_1),C)$ and $Sat(L(r_2),C)$ are different knots for $r_1\neq r_2$.
\end{thm}

\begin{proof}
We begin by supposing that the parity of $r_1$ and $r_2$ differ, i.e., $r_1\;mod\;2\neq r_2\;mod\;2$. Then, the result is immediate since their degree differ. Suppose that $r_1$ is odd, and that $r_2$ is even. This would lead to $\Delta_{Sat(L(r_1),C)}=1$ and $\Delta_{Sat(L(r_2),C)}=\Delta_C(t^2)$. However, since $\Delta_C(t)\neq 1$, we have $\Delta_{Sat(L(r_1),C)}\neq \Delta_{Sat(L(r_2),C)}$, and consequently $Sat(L(r_1),C)$ and $Sat(L(r_2),C)$ are also different.

Now, let us suppose that they have the same parity. We will compare the Jones polynomial of two consecutive lassos, $L(r)$ and $L(r+2)$, for which purpose we first consider the difference of their polynomials in $ST$.
\begin{align*}
J\subST(L(r))&-J\subST(L(r+2))=\\
&=(-t)^{-r}z^0\subST-t^{^1/_2}\frac{t^r-(-1)^{r}}{t^r(t+1)}z^2\subST-\bigg[(-t)^{-r-2}z^0\subST-t^{^1/_2}\frac{t^{r+2}-(-1)^{r}}{t^{r+2}(t+1)}z^2\subST\bigg]\\
&=(-1)^r\bigg(\frac{1}{t^r}-\frac{1}{t^{r+2}}\bigg)z^0\subST-t^{^1/_2}\bigg[\frac{t^r-(-1)^{r}}{t^r(t+1)}-\frac{t^{r+2}-(-1)^{r}}{t^{r+2}(t+1)}\bigg]z^2\subST\\
&=(-1)^r\bigg(\frac{t^2}{t^{r+2}}-\frac{1}{t^{r+2}}\bigg)z^0\subST-t^{^1/_2}\bigg[\frac{t^r-(-1)^{r}}{t^r(t+1)}-\bigg(\frac{t^r-(-1)^{r}}{t^{r+2}(t+1)}+\frac{(t^2-1)t^r}{t^{r+2}(t+1)}\bigg)\bigg]z^2\subST
\end{align*}
\begin{align*}
&=(-1)^r\frac{t^2-1}{t^{r+2}}z^0\subST-t^{^1/_2}\bigg[(t-1)\frac{t^r-(-1)^{r}}{t^{r+2}}-\frac{(t-1)t^r}{t^{r+2}}\bigg]z^2\subST\\
&=(-1)^r\frac{t^2-1}{t^{r+2}}z^0\subST-t^{^1/_2}(t-1)\frac{-(-1)^{r}}{t^{r+2}}z^2\subST=\frac{(-1)^r(t-1)}{t^{r+2}}\big[(t+1)z^0\subST+t^{^1/_2}z^2\subST\big].
\end{align*}

These polynomials will then be equal only if this difference is $0$. In other words, $J\subST(L(r))=J\subST(L(r+2))$ if and only if $(t+1)z^0\subST+t^{^1/_2}z^2\subST=0$, since we are on the hypothesis that $|t|>1$. This last equality translates into $z^2\subST=(-t^{^1/_2}-t^{^{-1}/_2})z^0\subST$. Hence, when considering the satellites of these knots with the companion $C$ and repeating the calculus for the satellite case using \emph{Theorem \ref{main}}, the condition for $J(Sat(L(r),C))$ and $J(Sat(L(r+2),C))$ to be distinct is as follows:
\[J(C;2)\neq-t^{^1/_2}-t^{^{-1}/_2}=J(U;2),\]
where $U$ is the unknot, as denoted in the previous section.

This condition holds if we extend the result to any arbitrary $r_1$ and $r_2$ having these the same parity (not necessarily being consecutive). It is enough to apply this result to the partial differences by two between $r_1$ and $r_2$ repeatedly to get the result.
\end{proof}


Now, if one considers $J(Sat(L(r),C))$ as a series of polynomials on $t$ where $r$ is the index, one can even calculate its limit, and we get the next result.

\begin{lmm}
If $|t|>1$, we obtain the following limit.
\[\lim\limits_{r\rightarrow\infty}J(Sat(L(r),C))=\frac{J(C;2)}{J(U;2)}.\]
\end{lmm}

\begin{proof}
We first calculate the limit of the Jones polynomials in $ST$ of the lasso family.
\[\lim\limits_{r\rightarrow\infty}J\subST(L(r))=\lim\limits_{r\rightarrow\infty}\bigg[\bigg(\frac{-1}{t}\bigg)^{r}z^0\subST-t^{^1/_2}\frac{1-(-1/t)^r}{t+1}z^2\subST\bigg]=-\frac{t^{^1/_2}}{t+1}z^2\subST=\frac{z^2\subST}{-t^{^1/_2}-t^{^{-1}/_2}}.\]
Then, we proceed to make use of it.
\[\lim\limits_{r\rightarrow\infty}J(Sat(L(r),C))=\lim\limits_{r\rightarrow\infty}J\subST(L(r))\Big|_{z^k\subST=J(C;\:k)}=\frac{z^2\subST}{-t^{^1/_2}-t^{^{-1}/_2}}\Big|_{z^k\subST=J(C;\:k)}=\frac{J(C;2)}{J(U;2)}.\]
\end{proof}

\begin{nrmrk}
It is interesting to notice that if $C$ is the trivial knot, the limit equals to $1$, which, on the other hand, is obvious since the satellite knots of lassos through the trivial knot are trivial again. This also holds for any knot (whether it exists or not) whose 2-parallel Jones polynomial equals to the one of the unknot, but that was also predictable since the Jones polynomial of a satellite knot of a lasso with such knot is, in the first place, also constantly $1$ (see equation \ref{L(r)} and apply it to construct a satellite knot with such $C$). 
\end{nrmrk}

\subsection{$L(1,r)$: lassos of degree 1 or 3}
In this subsection we will make the same assumptions as in the previous one. Let us start by giving the basic results regarding the Kauffman bracket and Jones polynomial in $ST$ of the family $L(1,r)$.
\begin{align*}
&\kauf{L(1,r)}=(-1)^rA^{1-3r}z^1\subST+A^{-1}z^3\subST\kauf{L(r)},\\
&T(wr(L(1,r)))=T(-1-r),
\end{align*}
where, again, $T(n)=(-A^{-3})^n$. Next, we give the formula for the Jones polynomial, which is also proven by using \emph{Proposition \ref{lasso-prop}}.
\[J\subST(L(1,r))=\big(-t^{-1}+(-t)^{-r}(t^{-1}+1)\big)z^1\subST+z^3\subST\frac{1-(-t)^{-r}}{t+1}.\]

Suppose now that $|t|>1$. Then we have the next result.

\begin{thm}\label{lasso-thm2}
Consider the family of lassos $L(1,r)$ and let $C$ be a knot in $\mathbb{S}^3$ such that $\Delta_C(t)\neq 1$ and $J(C;3)\neq J(C)J(U;3)$. Then, $Sat(L(1,r_1),C)$ and $Sat(L(1,r_2),C)$ are different knots for $r_1\neq r_2$.
\end{thm}

\begin{proof}
The modus operandi is exactly the same as in the previous theorem. As before, if the parity of $r_1$ and $r_2$ differ, the Alexander polynomials of the satellite knots arising from these lassos would equal one to $\Delta_C(t)$ and the other to $\Delta_C(t^3)$. Since we assumed that $\Delta_C(t)\neq 1$, then those polynomials (hence those satellites) are distinct.

It remains now to analyze the case in which $r_1$ and $r_2$ have the same parity. Operating as before, the result for the difference that one gets is the following.
\[J\subST(L(1,r))-J\subST(L(1,r+2))=\frac{(-1)^r(t-1)}{t^{r+2}}\bigg[\frac{(t+1)^2}{t}z^1\subST-z^3\subST\bigg].\]
Hence giving the condition below for $J(Sat(L(1,r),C))$ and $J(Sat(L(1,r+2),C))$ to be distinct.
\[J(C;3)\neq J(C)(-t^{^1/_2}-t^{^{-1}/_2})^2=J(C)J(U;3).\]
The validity of this condition also extends to arbitrary $r_1$ and $r_2$ ---by the same procedure as above---, and therefore the result is obtained.
\end{proof}


We can also calculate the limit of this family of polynomials on $r$.

\begin{lmm}
If $|t|>1$, we obtain the following limit.
\[\lim\limits_{r\rightarrow\infty}J(Sat(L(1,r),C))=-\frac{1}{t}J(C)+\frac{1}{t+1}J(C;3).\]
\end{lmm}

\begin{proof}
As before, we conduct an previous inspection on the limit of the Jones polynomial in $ST$ of $L(1,r)$.
\begin{align*}
\lim\limits_{r\rightarrow\infty}J\subST(L(1,r))&=\lim\limits_{r\rightarrow\infty}\bigg[\bigg(-\frac{1}{t}+\bigg(\frac{-1}{t}\bigg)^{r}\bigg(\frac{1}{t}+1\bigg)\bigg)z^1\subST+z^3\subST\frac{1-(-1/t)^r}{t+1}\bigg]=-\frac{1}{t}z^1\subST+\frac{1}{t+1}z^3\subST.
\end{align*}

Then by substituting when the satellite is taken, we get the result.
\end{proof}

\begin{nrmrk}
Just as in the previous case, the limit of the polynomial of a satellite knot with the unknot as companion is 1, since the Jones polynomial in $ST$ of $L(1,r)$ is also in the first place $1$.
\end{nrmrk}





\subsection{Practical example}
In regard to these last results, let us make an easy check on them in the following example.
\begin{exmp}
Using the lassos $L(2)$ (\emph{Whitehead double}) and $L(4)$ as patterns, let us build their satellite knots where the companion of both is the left-handed trefoil $3_1$. Using the fact that $L(2),L(4)\in \mathscr{L}_0$, one has, by \emph{Proposition \ref{alex}}, that the Alexander polynomials of their satellite knots are trivial.
\[\Delta_{Sat(L(2),3_1)}(t)=\Delta_{Sat(L(4),3_1)}(t)=1.\]

Now, here is where \emph{Theorem \ref{main}} comes to scene together with equation \ref{L(r)}. We will show practically how the Jones polynomial of the satellite knots is now easily computed.

First, we calculate the 2-parallel Jones polynomial of $3_1$. In order to do so, I used the software of \cite{Ko}.
\begin{align*}
&\spa{-2}J(3_1;\:2)=-t^{^{-23}/_2}+t^{^{-21}/_2}+t^{^{-17}/_2}-t^{^{-9}/_2}-t^{^{-5}/_2}-t^{^{-1}/_2}.&&&&&&&&
\end{align*}

Next, using equation \ref{L(r)}, we extract the Jones polynomial in $ST$ of $L(2)$, and then we proceed to apply \emph{Theorem \ref{main}}.

\begin{align*}
&J\subST(L(2))=t^{-2}z^0\subST-t^{^1/_2}\frac{1-t^{-2}}{t+1}z^2\subST.\\
&J(Sat(L(2),3_1))=\bigg(t^{-2}z^0\subST-t^{^1/_2}\frac{1-t^{-2}}{t+1}z^2\subST\bigg)\bigg|_{z^k\subST=J(3_1;\:k)}=t^{-2}-t^{^1/_2}\frac{1-t^{-2}}{t+1}J(3_1;\:2)\\
&\qquad =t^{-2}+t^{-13}-2t^{-12}+t^{-11}-t^{-10}+t^{-9}+t^{-6}-t^{-5}+t^{-4}-t^{-3}+t^{-2}-t^{-1}\\
&\qquad =t^{-13}-2t^{-12}+t^{-11}-t^{-10}+t^{-9}+t^{-6}-t^{-5}+t^{-4}-t^{-3}+2t^{-2}-t^{-1}.
\end{align*}

We repeat now the same process with $L(4)$.
\begin{align*}
&J\subST(L(4))=t^{-4}z^0\subST-t^{^1/_2}\frac{1-t^{-4}}{t+1}z^2\subST.\\
&J(Sat(L(4),3_1))=\bigg(t^{-4}z^0\subST-t^{^1/_2}\frac{1-t^{-4}}{t+1}z^2\subST\bigg)\bigg|_{z^k\subST=J(3_1;\:k)}=t^{-4}-t^{^1/_2}\frac{1-t^{-4}}{t+1}J(3_1;\:2)\\
&\qquad =t^{-4}+t^{-15}-2t^{-14}+2t^{-13}-3t^{-12}+2t^{-11}-t^{-10}+t^{-9}+t^{-8}-t^{-7}+2t^{-6}\\
&\qquad -2t^{-5}+2t^{-4}-2t^{-3}+t^{-2}-t^{-1}\\
&\qquad =t^{-15}-2t^{-14}+2t^{-13}-3t^{-12}+2t^{-11}-t^{-10}+t^{-9}+t^{-8}-t^{-7}+2t^{-6}-2t^{-5}\\
&\qquad +3t^{-4}-2t^{-3}+t^{-2}-t^{-1}.
\end{align*}

We see that, indeed, their polynomials are different, which is the result that we knew beforehand since $\Delta_{3_1}(t)\neq 1$ and $J(3_1;\:2)\neq J(U;2)$, as seen in \emph{Theorem \ref{lasso-thm1}}.
\end{exmp}


\newpage
\section{Final remarks}
What have we got in the end? A definition of a new kind of knots ---lassos--- inside the solid torus, a powerful result regarding the Alexander polynomial of satellite knots using lassos, and an explicit formula for calculating the Jones polynomial of any satellite knot. Plus, it has been also proven in this thesis that, indeed, the constructions of satellite knots using lassos are essentially distinct for certain families of lassos thanks to the Jones polynomial.

We can also create ``by request'' knots that have the same Alexander polynomial as any other given knot with $t$ taken to the powers of $0$, $1$, $2$ or $3$, being sure that the knot we create is essentially different from the originally given.

In the following lines we will serve ourselves from this well-known result that I mentioned in the Introduction of this paper, and that is an immediate consequence of \emph{Theorem \ref{lick-thm}}.
\[\Delta_{K_1\#K_2}(t)=\Delta_{K_1}(t)\Delta_{K_2}(t).\]

This result expresses the Alexander polynomial of a connected sum of knots in terms of the Alexander polynomials of its components. This is a consequence of the previously mentioned theorem for the Alexander polynomial of satellite knots, since a connected sum is a special case of satelliting. By using this result, now we can easily be told:
\[\textnormal{``I want a knot $K$ whose Alexander polynomial is $\Delta_{K}(t)=\Delta_{5_1}(t)^2\Delta_{8_{19}}(t^3)$''}.\]

We would proceed to build it with a smile on our face. For us, in this case it suffices to use a lasso $L\in\mathscr{L}_3$ ---since there is a $t^3$--- and then claim that such knot could be the following.
\[K=5_1\:\#\:5_1\:\#\:Sat(L(1,1),8_{19}).\]

But if we felt generous, we could also give an infinite amount of examples, where we would only change the lasso to any other of the same degree so that the Alexander polynomial would remain the same. For example, we could set the knot $K$ as the one below.
\[K=5_1\:\#\:5_1\:\#\:Sat(L(-3,7),8_{19}).\]

Or even boast using the same results to present other knots combinations such as:
\[K=5_1\:\#\:5_1\:\#\:Sat(L(1,1),8_{19})\:\#\:Sat(L(2),10_{161}).\]

In the last part of the paper we just proved that the satellites built using lassos are distinct from each other up to degree $3$, but there is no doubt that this can also be generalized to any arbitrary degree (I am currently working on its proof).\\

Let me now finish this summary by recapitulating with a beautiful and simple example. Let us consider the apple of my eyes, the lasso $L(1,2)$. It happens to be the simplest lasso of degree $1$. Consequently, by using it repeatedly we are able to build knots with the exact same Alexander polynomial as a given knot, without further constructions than a simple satellite composition. So given a knot $K$, we have the next equality.
\[\Delta_{Sat(L(1,2),K)}(t)=\Delta_{K}(t).\]

Using this fact, higher compositions can also be considered:
\[\Delta_{K}(t)=\Delta_{Sat(L(1,2),K)}(t)=\Delta_{Sat(L(1,2),Sat(L(1,2),K))}(t)=...\;.\]

But focusing on the first equality, and making $K$ the simplest knot known to man ---the trefoil--- a depiction of the knot $Sat(L(1,2),3_1)$, that has the same Alexander polynomial as the trefoil but different Jones polynomial, is as follows.

\begin{align*}
&\includegraphics[scale=0.3]{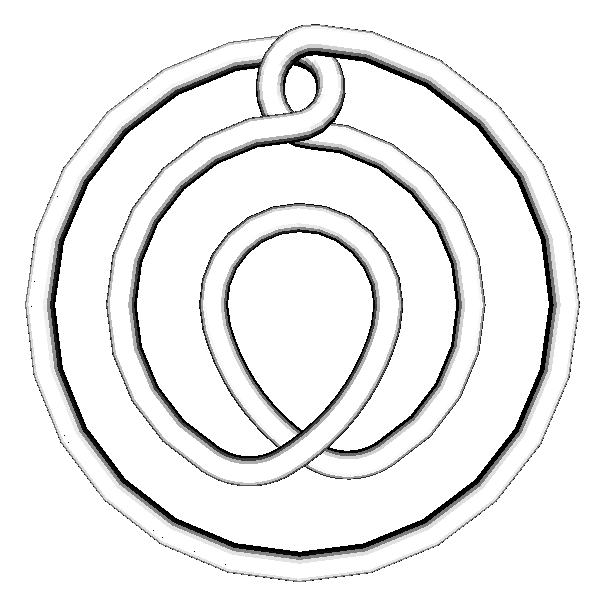}\spa{20}\includegraphics[scale=0.36]{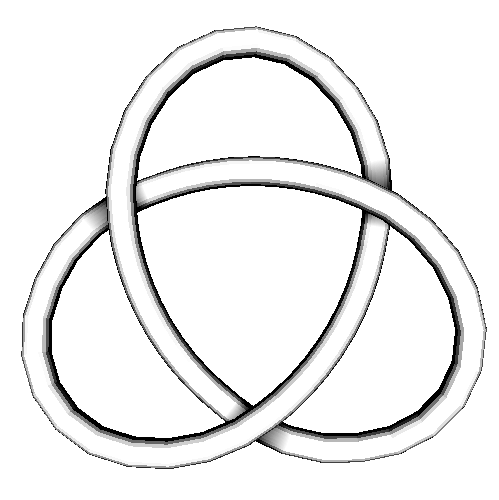}\\
&\spa{20}\spa{11}L(1,2)\spa{20}\spa{20}\spa{20}\spa{20}\spa{6}3_1\\
&\spa{20}\spa{20}\spa{20}\spa{20}\includegraphics[scale=1]{Arrow_down}\\
&\spa{20}\spa{10}\includegraphics[scale=0.36]{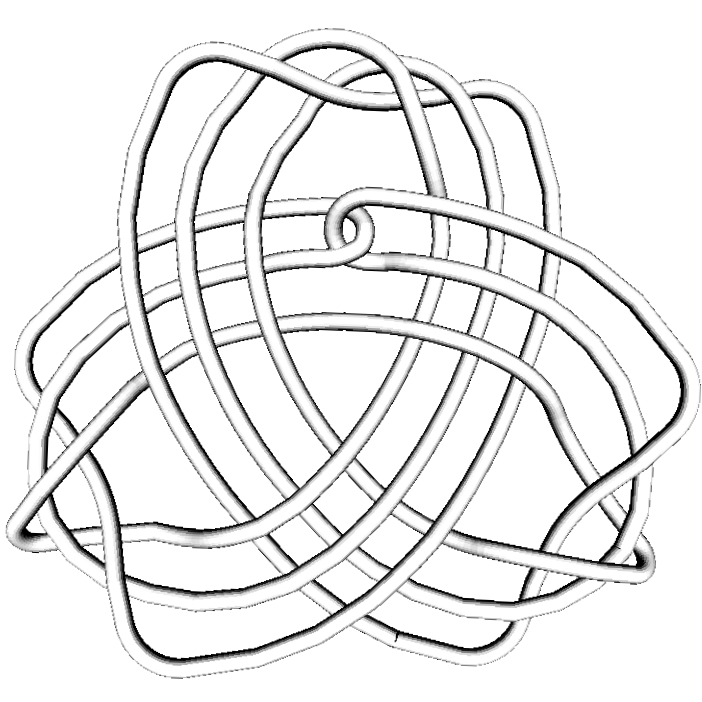}\\
&\spa{20}\spa{20}\spa{20}\spa{11}Sat(L(1,2),3_1)
\end{align*}

\newpage

\section*{Etymology}
The word \emph{lasso} is defined (by the Oxford dictionary) as ``a rope with a noose at one end, used especially in North America for catching cattle''. Its origin dates back to the beginning of the nineteenth century, coming from the Spanish word \emph{lazo}, from Latin \emph{laqueum}, also related to \emph{lace}.\\ 
In a sense, the lassos here defined ``catch'' the center of the torus and ``immobilize'' it. But \emph{torus} is a word also coming from Latin that degenerated into \emph{toro} in Spanish, in which, in turn, happens to mean \emph{bull} as well. Therefore the simile is round, and there was no better name for lassos than lassos.


\noindent Graduate School of Mathematical Sciences, \textsc{The University of Tokyo}\\
\emph{E-mail address}: \verb|adri@ms.u-tokyo.ac.jp|
\end{document}